\documentclass[12pt,reqno]{amsart}
\usepackage[french, english]{babel}
\usepackage[active]{srcltx}%%%% Inverse search
\usepackage{amsmath, amstext, amsthm, amssymb, amsfonts, latexsym}
\usepackage{graphicx}
\usepackage{graphics}
\usepackage[colorlinks=true,linkcolor=red,urlcolor=black]{hyperref} 
\usepackage{bbm} % para letras bbm como en Gibbs
\usepackage[pagewise]{lineno}
\usepackage{enumitem}

\newtheorem{Theorem}{Theorem}[section]
\newtheorem{prop}[Theorem]{Proposition}

\newtheorem{lemma}[Theorem]{Lemma}

\newtheorem{maintheorem}{Theorem}
\newtheorem{mainlemma}{Lemma}

\newcommand{\bmt}{\begin{maintheorem}}
\newcommand{\emt}{\end{maintheorem}}
\newcommand{\bml}{\begin{mainlemma}}
\newcommand{\eml}{\end{mainlemma}}

\theoremstyle{remark}
\newtheorem{remark}{Remark}[section]
\numberwithin{equation}{section}

\newcommand{\Diff}{{\rm Diff}^{r}} %difeos

\newcommand{\ME}{\mathcal{U}_{\mathcal{ME}}} %mostly expanding
\newcommand{\PH}{\mathcal{PH}} %\partial hyperbolic
\newcommand{\Gu}{{\mathbb{G}\mathrm{ibbs^u}}} %u_Gibbs
\newcommand{\Gcu}{{\mathbb{G}\mathrm{ibbs^{cu}}}} %cu-Gibbs
\newcommand{\M}{\mathbb{M}} %measures
\newcommand{\la}{\lambda}

\newcommand{\N}{\mathbb{N}}
\newcommand{\R}{\mathbb{R}}

\newcommand{\supp}{{\rm supp}\:}

\newcommand{\cU}{\mathcal{U}}
\newcommand{\cK}{\mathcal{K}}
\newcommand{\cQ}{\mathcal{Q}}

\newcommand{\cW}{\mathcal{W}^{cu}}
\newcommand{\cC}{\mathcal{C}}

\newcommand{\vol}{\operatorname{vol}}
\newcommand{\Bl}{\operatorname{Bl}}
\newcommand{\Emb}{\operatorname{Emb}}
\newcommand{\di}{\operatorname{dist}}

\title[Statistical Stability.]{Statistical stability of mostly expanding diffeomorphisms}
\date{\today}
\author[Martin Andersson]{Martin Andersson}
\address{Martin Andersson, Universidade Federal Fluminense, Departamento de Matem\'{a}tica Aplicada,
Rua Professor Marcos Waldemar de Freitas Reis, s/n, 24210-201, Niter\'{o}i, Brasil}
\email{nilsmartin@id.uff.br}
\author[Carlos H. V\'asquez]{Carlos H. V\'asquez}
\address{Carlos H. V\'asquez,
Instituto de Matem\'atica,
Pontificia Universidad Cat\'olica de Valpara\'{\i}so,
Blanco Viel 596, Cerro Bar\'on, 
Valpara\'{\i}so-Chile.
}
\email{carlos.vasquez@pucv.cl}
\begin{thanks}{C.V. was partially supported by Fondecyt 1171427,  CONICYT, Chile.}
\end{thanks}



\newcommand{\newabstract}[1]{%
  \par\bigskip
  \csname otherlanguage*\endcsname{#1}%
  \csname captions#1\endcsname
  \item[\hskip\labelsep\scshape\abstractname.]
}

\begin{document}

\begin{abstract}
We study how physical measures vary with the underlying dynamics in the open class of $C^r$, $r>1$, strong partially hyperbolic diffeomorphisms for which the central Lyapunov exponents of every Gibbs $u$-state is positive. 
If transitive, such a diffeomorphism has a unique physical measure that persists and varies continuously with the dynamics.  

A main ingredient in the proof is a new Pliss-like Lemma which, under the right circumstances, yields frequency of hyperbolic times close to one. Another novelty is the introduction of a new characterization of Gibbs $cu$-states. Both of these may be of independent interest.

The non-transitive case is also treated: here the number of physical measures varies upper semi-continuously with the diffeomorphism, and physical measures vary continuously whenever possible.

\newabstract{french}

Nous \'etudions comment les mesures physiques varient avec la dynamique sous-jacente, dans la classe ouverte des diff\'eomorphismes $C^r$, $r>1$, fortement partiellement hyperboliques pour lesquelles les exposants de Lyapunov centraux de tout $u$-\'etat de Gibbs sont positifs. Lorsque transitifs, de tels diff\'eomorphismes poss\'edent une unique mesure physique qui persiste et varie continûment avec la dynamique.

Un des ingr\'edients principaux de la preuve est un nouveau lemme de type Pliss qui, appliqu\'e dans le contexte ad\'equate, implique que la fr\'equence des temps hyperboliques est proche de un. Une autre nouveaut\'e est l'introduction d'une nouvelle caract\'erisation des $cu$-\'etats de Gibbs. Chacun de ses deux aspects ayant leur propre int\'er\^et.

Le cas non transitif est aussi trait\'e : dans ce contexte, le nombre de mesures physiques est une fonction semi-continue sup\'erieure du diff\'eomorphisme, et les mesures physiques varient continument sous des hypoth\`eses naturelles.

\end{abstract}

\maketitle

%\newpage
%\tableofcontents

\newpage

%%%%%%%%%%%%%%%%%%%%%%%%%%%%%%%%%%%%%%%%%%%%%%%%%%%%%%%%%%%%%%%%%%%%%%%%%%%%%

\section{Introduction} \label{sec:intro}

The present work deals with the question of continuity of physical measures in the setting of partially hyperbolic diffeomorphisms, i.e., with a uniformly contracting bundle $E^s$, a central bundle $E^c$, and a uniformly expanding bundle $E^u$. Here, by {\em physical measure} we mean a Borel probability $\mu$ for which the basin
\begin{equation*}
B(\mu) = \{x \in M : \frac{1}{n} \sum_{k=0}^{n-1} \delta_{f^k(x)} \rightarrow \mu \}
\end{equation*}
has positive Lebesgue measure. 

The particular context under consideration here is that in which the central direction is mostly expanding, by which we mean that every Gibbs $u$-state has positive central Lyapunov exponents. Here, by {\em Gibbs $u$-states} we mean invariant probabilities absolutely continuous with respect to Lebesgue measure along the partition in strong unstable manifolds.

Mostly expanding diffeomorphisms include some very important examples from smooth ergodic theory, among which we have Derived-from-Anosov diffeomorphisms \cite{AV2018} as well as the well-known examples of partially hyperbolic diffeomorphisms with pathological center foliation due to Shub and Wilkinson (see \cite{SW2000} and \cite[Theorem~D]{Y2016}). The above definition of mostly expanding partially hyperbolic diffeomorphisms was introduced in \cite{AV2018}. It is important to note that it differs from the homonymous notion considered in the influential work \cite{ABV2000}. The reason for this change in terminology is that the concept of mostly expanding introduced in \cite{AV2018} (and considered here) is a closer analogue of the notion of \emph{mostly contracting}, introduced in \cite{BV2000}. The original notion of mostly expanding used in \cite{ABV2000} will be referred to by us as  {\em non uniformly expanding along the  center-unstable direction}, or simply the {\em NUE-condition}. As  showed in  \cite{AV2018, Y2016}, a partially hyperbolic diffeomorphism which is mostly expanding will necessarily satisfy the NUE-condition. That is,

$$ \mbox{Mostly expanding}\implies\mbox{ NUE-condition}$$
(but not the other way around). Another important feature of the mostly expanding condition is that it is open in the $C^r$ topology for $r>1$ --- something that is not true for the more general NUE-condition (see \cite{AV2018} for proofs and counter-examples). 
The regularity $r>1$ here may be any real number. If $r$ is not an integer, i.e. $r = k+\alpha$ for some $k \in \N$ and $\alpha \in (0,1)$, then by $C^r$ diffeomorphism we mean a $C^k$ diffeomorphism whose $k$-th derivative is $\alpha$-H\"older, and the topology considered is the one induced by a metric which, in charts, can be written as the sum of the $C^k$ metric and the $\alpha$-H\"older metric on the $k$-th derivative.

The NUE-condition was introduced as a way to guarantee the existence and finitude of physical measures. Since the mostly expanding condition is open and implies the NUE-condition, it provides an ideal setting for the following much sought after situation:

\begin{Theorem}[\cite{AV2018}] \label{MACA1teoC}
Let $f:M\to M$ be a $C^r$, $r>1$, partially hyperbolic  diffeomorphism  on a compact manifold. Suppose that every Gibbs $u$-state of $f$ has positive central Lyapunov exponent (i.e. $f$ is mostly expanding). Then there exists a $C^r$ neighborhood $\cU$ of $f$ such that every $g \in \cU$ is mostly expanding and has a finite number of physical measures whose basins together cover a full Lebesgue measure set in $M$.
\end{Theorem}

Theorem \ref{MACA1teoC} begs the question of how the physical measures of a mostly expanding diffeomorphism react to small deterministic perturbations. Will the number of physical measures remain the same? If so, do they vary continuously with the dynamics? Such questions have been answered quite satisfactorily for the analogous notion of mostly contracting diffeomorphisms. A partially hyperbolic diffeomorphism is said to be \emph{mostly contracting} if all its Gibbs $u$-states have negative central Lyapunov exponents. Just like the mostly expanding condition, it is $C^r$ open for any $r>1$. (This was noted in \cite{D2004} and proved in detail in \cite{A2010}).  Moreover it was proved in \cite{A2010} (see also \cite{DVY2016}) that 
\begin{enumerate}
\item the number of physical measures of mostly contracting diffeomorphisms varies upper semi-continuously with the dynamics, and
\item physical measures vary continuously in the weak* topology under perturbations that don't change the number of physical measures.
\end{enumerate}
In \cite{VY2013,DVY2016} the authors gave a detailed explanation of how bifurcations occur and an exhaustive set of examples.  They have also proved a form of continuity of the basins of physical measures.

The aim of this work is to complete the picture by proving similar results for mostly expanding diffeomorphisms. We start by considering the case in which the physical measure is unique.

\begin{maintheorem}\label{one pm}
Let $f:M \to M$ be a $C^r$ , $r>1$, transitive partially hyperbolic diffeomorphism of type $TM = E^s \oplus E^c \oplus E^u$ such that every Gibbs $u$-state has positive central Lyapunov exponents. Then there is a $C^r$ neighborhood $\cU$ of $f$ such that every $g \in \cU$ has a unique physical measure $\mu_g$. Moreover $\mu_g$ varies continuously with $g$ in the weak* topology.
\end{maintheorem}

To our knowledge, there are no known examples of mostly expanding diffeomorphisms with more than one physical measure. Yet we can show that if such examples exist (and we believe they do), then their physical measures behave just like in the mostly contracting case:

\begin{maintheorem}\label{several pm}
Let $f:M \to M$ be a $C^r$, $r>1$, partially hyperbolic diffeomorphism of type $TM = E^s \oplus E^c \oplus E^u$ (not necessarily transitive) such that every Gibbs $u$-state has positive central Lyapunov exponents. Then the number of physical measures depends upper semi-continuously on $g$ and physical measures vary continuously in the weak* topology on any subset $\cC \subset \cU$ on which the number of physical measures is constant.
\end{maintheorem}

Theorem~\ref{one pm} is  in fact a corollary of Theorem~\ref{several pm}, but it is by far the case of greatest interest and therefore deserves to be in the spotlight.

In spite of the strong analogy between the notions of mostly expanding and mostly contracting, there is at least one important difference between the two. In the former case,   the basin of a physical measure is an open set, modulo a Lebesgue null set \cite[Lemma 4.5]{AV2018}. That is very different from the case of mostly contracting, in which one may have the phenomenon of intermingled basins of attraction \cite{K1994}. It also means that, for mostly expanding diffeomorphisms, transitivity is sufficient to guarantee uniqueness of the physical measure.

Some comments on terminology are pertinent. By statistical stability one usually means a situation where all physical measures persist and vary continuously with small perturbations on the dynamics. One can therefore say that a diffeomorphism satisfying the hypotheses of Theorem~\ref{one pm} is statistically stable. For the situation in Theorem~\ref{several pm} one can likewise talk about statistical stability within some suitable parameter space (in which the number of physical measures remain constant). On the other hand, it is possible to weaken the notion of statistical stability, saying that a diffeomorphism $f:M \to M$ is $C^r$-\emph{weakly statistically stable} if, given any neighborhood $U$ of the closed convex hull of the physical measures of $f$,  there exists a $C^r$ neighborhood $\mathcal{U}$ of $f$ such that, given any $g \in \cU$, every physical measure of $g$ belongs to $U$. To appreciate the difference between statistical stability and weak statistical stability, it is instructive to look at the examples from \cite[Section 3.3]{DVY2016}. There,  the authors produce examples of mostly contracting diffeomorphisms for which the number of physical measures decreases under arbitrarily small perturbations. The closed convex hull of the physical measures of the perturbed system lie in a neighborhood of a proper subsymplex of the closed convex hull of the physical measures of the unperturbed system. The unperturbed system is therefore weakly statistically stable but not statistically stable.

\begin{maintheorem} \label{weak stability}
Let $f:M \to M$ be a $C^r$, $r>1$, partially hyperbolic diffeomorphism of type $TM = E^s \oplus E^c \oplus E^u$ such that every Gibbs $u$-state has positive central Lyapunov exponents. Then $f$ is weakly statistically stable.
\end{maintheorem}

In a recent paper \cite{Y2016}, Yang proved that having positive central Lyapunov exponents with respect to every Gibbs $u$-state is a $C^1$ open property. It is therefore natural to ask whether the physical measures vary continuously with the dynamics in the $C^1$ topology. We do not know.

In another recent work \cite{MCY2017}  the authors prove  existence and finiteness of physical measures for partially hyperbolic diffeomorphisms $f$ with dominated splitting $ TM=E^{cs}\oplus E^{cu}\oplus E^u$ (with uniform expansion in $E^u$) satisfying a mixture of {\em mostly contracting} and mostly expanding behavior. That is, every Gibbs $u$-state has positive Lyapunov exponents in the $E^{cu}$  bundle (i.e. $f$ is  mostly expanding along $E^{cu}$) and negative Lyapunov exponents in the $E^{cs}$ bundle (i.e. $f$ is mostly contracting along $E^{cs}$).  All results in the present work can be extended to the setting in \cite{MCY2017} using our results and following the ideas in the proof of  Theorem C in \cite{V2007}.   

There are several works on statistical stability related to our current setting, most notably \cite{A2000, AV2002, A2004} dealing with classes of non-invertible systems with non-uniformly expanding behavior. In \cite{A2000} the author proves the existence of measures absolutely continuous with respect to Lebesgue and then, \cite{AV2002, A2004} the authors give sufficient conditions to obtain an stronger version of statistical stability:  continuous variation (in $L^1$ norm) of densities of absolutely continuous invariant measures. Those works rely on techiques rather different from ours, namely the Ruelle transfer operator  and an induced piecewise expanding Markovian map. The latter approach, inspired by the works \cite{LSY1998, LSY1999}, consists of considering a disc $\Delta\subseteq M$  and a induced first return map $F_f:\Delta\to\Delta$. If  $F_f$ is piecewise expanding, then it has an invariant absolutely continuous measure $\mu_F$. If, in addition, the return times $R_f:\Delta \to \mathbb{N}$ of $f$ to $\Delta$ are integrable with respect to Lebesgue measure, then 
$$\mu_f=\sum_{j=0}^\infty f^j_*(\mu_F|\{R_f>j\})$$
is an absolutely continuous $f$-invariant finite measure. In order to obtain statistical stability, the author imposes a condition of uniformity (with respect to $f$) on the integral (with respect to Lebesgue) of the tail: given $\varepsilon>0$, there is $N=N(\varepsilon)\in\mathbb{N}$ such that for every $g$ close enough of $f$:
$$\sum_{j=N}^\infty\mathrm{Leb}(\{R_g>j\})<\epsilon.$$
A similar approach (and uniform condition on the tail) is used in \cite{ACF2010}. It is not clear to what generality such uniform tail conditions hold. All works cited above apply to robust classes of maps where the uniformity of tail behavior is implicit in the calculations that shows that they are non-uniformly expanding in the first place. Yet there doesn't seem to be any good reason why it would be intrinsic to non-uniform expansion, or that one can find a universal mechanism that gives rise to it.

In the setting of partially hyperbolic diffeomorphisms, a similar result was proved in \cite{V2007}. Briefly speaking,  consider a partially hyperbolic diffeomorphisms $f:M\to M$ with a $Df$-invariant splitting $TM=E^s\oplus E^c$ and satisfying the  NUE-condition on the center direction (here $E^{u}=\{0\}$ is taken trivial). In this case, as showed in  \cite{ABV2000}, there exists a finite number  of physical measures.  The author considers a sequence $f_n$  converging to $f$ in the $C^r$ topology, $r>1$. Since partially hyperbolicity is an open condition  (and the bundles vary continuously with the diffeomorphism), each $f_n$ also has a splitting of the tangent bundle $TM=E^s_n\oplus E^c_n$, $n\geq 1$. Assuming that $f_n$  satisfies a uniform  NUE-condition on the center direction, it is posible to choose a sequence $\mu_n$ of physical measures of each $f_n$ and consider an accumulation point $\mu$. Restricted to a ``fixed bounded cylinder'', the author proves  that $\mu$ is the sum of measures $\nu + \eta$, with $\nu$ non-zero, such that $\nu$ is a combination of physical measures of $f$ and $\eta$ is  like a ``singular rest''. This is done using a construction from \cite{ABV2000}, in which one shows that  $\mu_n$ can be approximated inside a ``bounded cylinder'' by a sum  $\nu_n + \eta_n$ with the mass $|\nu_n|$  of $\mu_n$ bounded away from zero (uniformly with respect to $n$), and such that $\nu_n$ has a disintegration along center-unstable manifolds with uniform bounds on the densities of its conditional measures. Therefore $\nu_n$ accumulates (in the cylinder) on a measure $\nu$ with the same properties and then it is in the convex hull of physical measure. By tacitly assuming uniformity of tail behavior of return maps to the cylinder, the author concludes that $\eta_n\to 0$ and  so  $\mu=\nu$. 
 
Our main finding in this work is that, by replacing the NUE-condition by the slightly stronger mostly expanding condition, we are able to rid ourselves completely of any assumptions about tail behavior. In fact, our strategy is essentially the same as in \cite{V2007}, but with the important improvement that $|\nu_n|$ can be taken to be not only bounded away from zero, but arbitrarily close to one.  The magic occurs because the following new version of the classical Pliss Lemma:

\begin{mainlemma}[{\bf Pliss-Like Lemma}]\label{Pliss-like lemma}
Let $L<\gamma < \Gamma$ and suppose that $a_1, \ldots a_N$ are numbers such that  $a_i\geq L$ for every $1\leq i \leq N$. Let $\kappa>0$ be a number such that 
\begin{equation*}
\#\{i\in \{1, \ldots, N\}: a_i < \Gamma \} \leq \kappa N
\end{equation*} 
and write $\theta =1- \kappa \frac{\Gamma-L}{\Gamma - \gamma}$. Then there exist $1 < n_1 < n_2< \ldots < n_m \leq N$, with $m \geq \theta N$, such that 
\begin{equation}\label{pliss condition}
\sum_{j=n+1}^{n_i} a_j \geq \gamma(n_i-n)
\end{equation}
for every $1 \leq i \leq m$ and every $0 \leq n <n_i$.
\end{mainlemma}

Both the Classical Pliss Lemma \cite{P1972} (see also \cite[Lemma 11.8]{M1987}  or \cite[Lemma 3.1]{ABV2000}) and Lemma~\ref{Pliss-like lemma} estimate how often (frequency $\theta$) certain partial averages of the finite sequence $a_i$ are close to its total average. In the Classical Pliss Lemma, $\theta$ is a function of an upper bound of the $a_i$ and its average. Unless these become close, $\theta$ cannot be taken close to $1$. In contrast, Lemma~\ref{Pliss-like lemma} considers a lower bound of the $a_i$ and a number $\Gamma$ slightly smaller than the average so that a good fraction ($1-\kappa$) of the numbers $a_i$ are larger than $\Gamma$. If $\kappa$ is small, then $\theta$ is close to one. That is the case whenever the sequence $a_i$ has small variation (most of its members are close to its mean) and the lower bound is not too far from the mean. In particular, it applies to Birkhoff sums of observables of the form 
\[ \varphi(x) = \frac{1}{\ell} \int \log \| Df^\ell \vert E_x^{c} \| d\mu,\]
where $f$ is a partially hyperbolic diffeomorphism, $\mu$ is an ergodic $f$-invariant measure, and $\ell$ is some large integer. Indeed,  Pliss' Lemma  has been intimately linked with the so called ``hyperbolic times''  introduced by Alves in \cite{A2000}. Loosely speaking, hyperbolic times are are numbers (times) associated to point whose orbit has non-uniformly hyperbolic behaviour. These times corresponding to iterates along the orbit for which the non-uniform hyperbolicity behaves as uniform hyperbolicity for large chunks of the orbit. Pliss' Lemma guarantees that an orbit on which a diffeomorphism is, say, asymptotically expanding in some direction, will have hyperbolic times on a set of iterates that correspond to positive frequency. Many of the difficulties related to hyperbolic times arise from the fact that the frequency of hyperbolic times provided by the Pliss' Lemma is only positive, but not necessarily close to one. This is in fact the main difficulty in the current work, and we overcome it by replacing the Pliss' Lemma by Lemma~\ref{Pliss-like lemma} and applying it to iterates of the diffeomorphism under consideration.

Another important ingredient are Gibbs $cu$-states and its properties. They were introduced in \cite{ABV2000} using the fact that, in the presence of positive Lyapunov exponents, there are Pesin invariant unstable manifolds. Thus Gibbs $cu$-states are a non-uniform version of  Gibbs $u$-states: While  the definition of Gibbs $u$-states involves the strong unstable foliation, we say that an invariant probability measure $\mu$ is a {\em Gibbs $cu$-state} if the conditional measures of $\mu$ along the corresponding local {\em unstable Pesin's manifolds} are almost everywhere absolutely continuous with respect to Lebesgue measure on these manifolds.  Gibbs $cu$-states are natural candidates to be physical measures. In fact, in our setting, a measure is physical if and only if  its is an ergodic Gibbs $cu$-state (see \cite[Lemma 4.4]{AV2018}). In Section~\ref{ssec:PH} we provide more details about the properties of  Gibbs $cu$-states, their relation to physical measures, and a complete toolbox to be used in our proof.

Here is an outline of our arguments.

\begin{enumerate}

\item Compactness of the set of Gibbs $u$-states (Proposition~\ref{prop:uM7}) provides us with uniform bounds, in a robust fashion, on the Lyapunov exponents of these (see \eqref{sigma}).

\item We use  Lemma~\ref{Pliss-like lemma}  to prove that an iterate of a mostly expanding diffeomorphism has hyperbolic times with frequency arbitrarily close to one (see Proposition~\ref{density of hyperbolic times}). This is a considerable improvement on the positive but possibly small frequency of hyperbolic times used in most arguments with a similar flavor. 

\item The abundance of hyperbolic times given by our Pliss-like Lemma is used to prove that, in our setting, limits of Gibbs $cu$-states are Gibbs $cu$-states (Theorem ~\ref{cu is closed}). This convergence is tricky to prove rigorously. We overcome this difficulty by introducing a useful characterization of Gibbs $cu$-states which does not directly involve disintegration of the measure (Theorem~\ref{characterization}).

\item Ergodic Gibbs $cu$-states are physical measures (Proposition~\ref{pm vs cu}).

\item Finally, distinct ergodic Gibbs $cu$-states cannot get too close to each others; therefore they must either stay apart or collapse into one ergodic Gibbs $cu$-states. This gives upper semi-continuity (see Section~\ref{proof several pm}). 
\end{enumerate}

As already mentioned above, Theorem~\ref{one pm} can be applied to a number of important examples. Certain Derived from Anosov diffeomorphisms, like the described in \cite[Section 6]{AV2018}, are mostly expanding with a unique physical measure, and therefore statistically stable.  Also Theorem~\ref{one pm} can be applied to generic $C^\infty$ perturbation $f$ of the time one map of a hyperbolic geodesic flow on a surface $M$. In this case, it was proved in \cite{D2003} that either $f$ of its inverse $f^{-1}$ is mostly expanding and so our results can be applied.  

Other family of examples where  our theorems are applied  are the examples provided by \cite[Theorem 1]{SW2000} . In this case, it  was proved in \cite{Y2016}  that if $f$ is  a $C^r$, $r>1$ accessible, volume preserving, partially hyperbolic diffeomorphism with one-dimensional center and the center exponent (with respect  to the volume measure) is positive, then it is mostly expanding.

%%%%%%%%%%%%%%%%%%%%%%%%%%%%%%%%%%%%%%%%%%%%%%%%%%%%%%%%%%%%%%%%%%%%%%%%%%%%%

\section{Some background}\label{sec:preliminaries}

\subsection{Dominated splitting and partial hyperbolicity}\label{ssec:PH}

Let $M$ be a  closed Riemannian manifold. 
We denote by $\|\cdot\|$ the norm obtained from the Riemannian structure 
and by $m$ the normalized volume measure on $M$ induced by the Riemannian structure. We often refer to $m$ as ``the Lebesgue measure on $M$". Moreover, if $D$ is a submanifold of $M$ we denote by $\vol_D$ the volume measure on $D$ induced by the Riemannian structure and by $m_D$ its normalization, i.e. $m_D=\vol_D /|\vol_D |$.

A diffeomorphism $f\::\:M\to M$  has a {\em dominated splitting} $F \prec G$ if there is a $Df$-invariant decomposition $TM= F\oplus G$ into complementary subbundles of $TM$ of constant dimensions, and $N \geq 1$ such that
\begin{equation*}\label{dd}
\|Df^N|F_x\|\cdot\|Df^{-N}|G_{f^N(x)}\| <1
\end{equation*}
for every $x\in M$. Any such splitting is necessarily continuous.

A diffeomorphism $f\colon M\rightarrow M$  is {\em partially hyperbolic} if there exists a  continuous $Df$-invariant splitting  $$T M=E^s\oplus E^c\oplus E^u,$$ such that $E^s \prec(E^c \oplus E^u)$ and $(E^s \oplus E^c) \prec E^u $ are both dominated splittings and, moreover, there exists $N \geq 1$ such that 
$\| Df^N \vert E^s \| <1$ and $\|Df^{-N} \vert E^u \| < 1$. We always assume that $\dim E^\sigma\geq1$, $\sigma=s,c,u$ unless stated otherwise. 

We denote by $\PH^{r}(M)$, $r\geq 1$, the set of $C^r$ partially hyperbolic diffeomorphisms defined on $M$. The set $\PH^{r}(M)$ is open in the $C^r$ topology and the bundles vary continuously with the diffeomorphism \cite[Corollary 2.17]{HP2006}. For partially hyperbolic diffeomorphisms, it is a well-known fact that there are foliations $\mathcal{F}^\sigma$ tangent to the distributions $E^\sigma$ for $\sigma=s,u$ \cite{HPS1977}. The leaf of $\mathcal{F}^{\sigma}$ containing $x$ will be called $W^{\sigma}(x)$, for $\sigma=s,u$.

\subsection{Gibbs $u$-states, Gibbs $cu$-states, and physical measures}\label{measure section}

In the following $f\colon M\rightarrow M$  is a $C^r$, $r>1$, partially hyperbolic diffeomorphism with $Df$-invariant splitting  $$T M=E^s\oplus E^c\oplus E^u.$$

Denote by  $\mathbb{M}^1(M)$ the set of Borel probability measures defined on $M$ provided with the weak* topology. We denote by $\mathbb{M}_f^1(M)$
the set of $f$-invariant probability measures. It is well known that $\mathbb{M}_f^1(M)$ is non empty, convex, closed (and so compact) subset of $\mathbb{M}^1(M)$.

\vspace{.5cm}
\noindent\textbf{Notation:} Throughout this work, the ``Riemannian topology'' of $M$, the $C^r$ topology on $\textrm{Diff}^r(M)$ and  the weak* topology on $\mathbb{M}^1(M)$ will be envolved. To avoid confusion we use 
\begin{itemize}
\item capital letters  $U, V,...$ to denote  subsets   of $M$,
\item caligraficral capital letters   $\mathcal{U}, \mathcal{V},...$ to denote subsets  of $\textrm{Diff}^r(M)$, 
\item board capital letters $\mathbb{U}, \mathbb{V},...$ to denote subsets of $\M^1(M)$, and 
\item boldface capital letters $\mathbf{U}, \mathbf{V}, \ldots$ to denote subsets of $\textrm{Diff}^r(M) \times \M^1(M)$.

\end{itemize}
\vspace{.5cm}

A measure   $\mu\in \mathbb{M}_f^1(M)$ is a {\em Gibbs $u$-state} if the conditional measures of $\mu$ with respect to the partition into local strong-unstable manifolds are absolutely continuous with respect to Lebesgue measure along the corresponding local strong-unstable manifold.  

We denote by $\Gu(f)$ the subset of  Gibbs $u$-states for $f$.  If $\mathcal{U}\subseteq \PH^r(M)$, we denote by
$$
\Gu(\mathcal{U}) := \{(g,\mu) : g \in \mathcal{U} \text{ and }\mu \in \Gu(g)\}
$$

For future reference, we list some relevant properties of Gibbs $u$-states. 

In the follows, if $g:M\to M$ is Borel measurable and $\nu$ is a Borel probability measure defined on $M$,  $g_*\nu$ denotes the Borel  probability measure  $\nu\circ g^{-1}$.

\begin{prop}\label{prop:uM1}{\bf[Pesin, Sinai; \cite{PS1982}]} If $f$ is a $C^r$ partially hyperbolic diffeomorphism, $r>1$, then there exists a Gibbs $u$-state. More precisely,  if $D$ is a disk of dimension $\dim (E^u)$ inside a strong unstable leaf, then every accumulation point of the sequence of probability measures 
$$\mu_n=\frac1n\sum_{k=0}^{n-1}f_*^k m_D $$
is a Gibbs $u$-state with densities with respect to the volume measure along the strong unstable leaves satisfying
\begin{equation}\label{eq:udensity}
\frac{\rho(x)}{\rho(y)} = \prod_{n=0}^{\infty} \frac{\det( Df^{-1} \vert E_{f^{-n}(x)}^{u})}{\det( Df^{-1} \vert E_{f^{-n}(y)}^{u})}.
\end{equation}
for any points $x,y$ in the same unstable plaque. As a consequence, the density $\rho$  along the strong unstable leaves are uniformly bounded away from zero and infinity.
\end{prop}

Clearly, convex combinations of Gibbs $u$-states are Gibbs $u$-states. Recall that if $\mu$ is any $f$-invariant measure, the limit 
$$\mu_x = \lim_{n \to \infty} \frac{1}{n} \sum_{k=0}^{n-1} \delta_{f^k(x)} $$
exists and is ergodic $\mu$-almost everywhere, and 
$$ \int \varphi \ d\mu = \int \left( \int \varphi \ d\mu_x \right) \ d\mu$$
for every continuous function $\varphi: M \to \mathbb{R}$.

\begin{prop}[{{\cite[Lemma 11.13]{BDV2005}}}] \label{prop:uM2}  Let $f:M \to M$ be a $C^r$ partially hyperbolic diffeomorphism,  $r>1$. If $\mu$ is a Gibbs $u$-state, then $\mu_x$ is a Gibbs $u$-state for $\mu$-almost every $x$. In other words, every Gibbs $u$-state $\mu$ is a convex combination of ergodic Gibbs $u$-states. 
\end{prop}

To complete our knowledge about the structure of the set of Gibbs $u$-state   we have,

\begin{prop}[{{\cite[Remark 11.15]{BDV2005}, \cite[Theorem 5]{BDPP2008}}}] \label{prop:uM7} 

Let $f:M \to M$ be a $C^r$ partially hyperbolic diffeomorphism,  $r>1$. Then the set $\Gu(f)\subseteq \M^1(M)$ is closed (and so compact)  and convex. Moreover, the map $\PH^r(M)\ni f\mapsto\Gu(f)\subset \mathbb{M}^1(M)$ is upper semicontinuous. 
\end{prop}

The last statement in the proposition above is equivalent to say  that,  given any sufficiently small $C^r$ neighborhood $\mathcal{U}$ of $f$, the set 
$\Gu(\mathcal{U})$ is closed in $\mathcal{U} \times \mathbb{M}^1(M)$.

Given a partially hyperbolic diffeomorphism $f: M \to M$, the {\em minimum central Lyapunov exponents} is the measurable function
\begin{align*}
\lambda^c(f,\cdot): M & \to \mathbb{R} \nonumber \\
x & \mapsto \liminf_{n \to \infty} \frac{1}{n} \log \| (Df^n \vert E_x^c )^{-1}\|^{-1} 
\end{align*}

We say that $f$ has {\em positive central Lyapunov exponents} with respect to the invariant measure $\mu$ if $\lambda^c(f,x)>0$ $\mu$-almost everywhere. We find it convenient to write $ E^{cu} = E^c \oplus E^u$. To say that $f$ has positive central Lyapunov exponents with respect to $\mu$ is then the same thing to say that 
\[ \liminf_{n \to \infty} \frac{1}{n} \log \| (Df^n \vert E_x^{cu} )^{-1}\|^{-1}>0 \]
$\mu$-almost everywhere. Indeed, it follows from the definition of partial hyperbolicity that  $\|Df^{-1}|E^u_x\|<\|Df^{-1}|E^c_x\|$ for every $x\in M$. Hence

$$ \liminf_{n \to \infty} \frac{1}{n} \log \| (Df^n \vert E_x^c )^{-1}\|^{-1} =\liminf_{n \to \infty} \frac{1}{n} \log \| (Df^n \vert E_x^{cu} )^{-1}\|^{-1}.$$

Recall that a $C^{r}$ partially hyperbolic diffeomorphism $f:M\to M$, $r>1$, is {\em mostly expanding} if all the Gibbs $u$-states of the diffeomorphism have positive central Lyapunov exponents. We denote by $\mathcal{U}_{\mathcal{ME}}\subset \PH^{r}(M)$ the set of $C^r$ mostly expanding partially hyperbolic diffeomorphism, with $r>1$. It was showed in \cite[Theorem B]{AV2018} that $\mathcal{U}_{\mathcal{ME}}$ is an open set in the $C^{r}$ topology, $r>1$.

 For our purpose, it is useful to consider  the {\em integrated minimum central Lyapunov exponent} defined by
\begin{align*}
\hat{\lambda}^c:\Gu(\mathcal{U}_{\mathcal{ME}}) & \to \mathbb{R} \nonumber \\
(g,\mu) & \mapsto \int  \lambda^c(g,x) \ d\mu(x) \label{def:lambdagorro}
\end{align*}

\begin{prop}[{{\cite[Proposition 3.4]{AV2018}}}] \label{semicont of Lyapunov exponent}
The function $\hat{\lambda}^c $ is lower semicontinuous.
\end{prop}

Following \cite{ABV2000}, if $f$ is mostly expanding,  we say that an $f$-invariant measure $\mu$  is a {\em Gibbs $cu$-state}  if $\mu$ has positive Lyapunov exponent and the conditional measures of $\mu$ along the corresponding   local (Pesin) center-unstable manifolds are almost everywhere absolutely continuous with respect to Lebesgue measure on these manifolds.
For each  $f\in \mathcal{U}_{\mathcal{ME}}$, we denote by $\Gcu(f)$ the subset of  Gibbs $cu$-states. If $\mathcal{U}\subseteq \mathcal{U}_{\mathcal{ME}}$, we denote by
$$
\Gcu(\mathcal{U}) := \{(g,\mu) : g \in \mathcal{U} \text{ and }\mu \in \Gcu(g)\}.
$$

Every Gibbs $cu$-state is in fact a Gibbs $u$-state with positive central Lyapunov exponents, although the converse is not true (see  the example in \cite[page 8]{AV2018}). We also have the following analogue of Proposition~\ref{prop:uM2}.

\begin{prop}[{\cite[Lemma 2.4]{V2007}}]\label{prop:cuM2}  Let $f:M \to M$ be a $C^r$ partially hyperbolic diffeomorphism,  $r>1$. Then every Gibbs $cu$-state $\mu$ is a convex combination of ergodic Gibbs $cu$-states. 
\end{prop}

It is rather straightforward to see that an ergodic Gibbs $cu$-satate must be a physical measure. In fact it is a special case of the situation treated in \cite{MR983869} (and which has become an underlying paradigm for much of smooth ergodic theory), in which it was shown that if an ergodic measure has only non-zero Lyapunov exponents and is absolutely continuous along Pesin's unstable manifolds, then it is a physical measure.   The next results says that, in the present context, the the converse is also true.

\begin{prop}[{{\cite[Lemma 4.4]{AV2018}}}] \label{pm vs cu}
Let $f:M \to M$ be a $C^r$ mostly expanding partially hyperbolic diffeomorphism, $r>1$.  Then the set of physical measures coincides with the set of ergodic Gibbs $cu$-states.
\end{prop}

A central result in this paper is the following Gibbs $cu$-states version of Proposition~\ref{prop:uM7}.

\begin{maintheorem}\label{cu is closed}
Let $\mathcal{U}_{\mathcal{ME}}$ be the  $C^r$ open set of mostly expanding diffeomorphisms, $r>1$. 
Then, the map $\mathcal{U}_{\mathcal{ME}}\ni f\mapsto\Gcu(f)\subset \mathbb{M}^1(M)$ is upper semicontinuous. 
\end{maintheorem}

Theorem~\ref{cu is closed} says that for every sequence $f_n $, $n\geq 1$, converging to $f\in\mathcal{U}_{\mathcal{ME}}$ in the $C^r$ topology, and every sequence $\mu_n\in \Gcu(f_n) $, all accumulation points of $\mu_n$  (in the weak* topology) belong to the set $\Gcu(f) $.
Notice that Theorem~\ref{cu is closed} is not a consequence of Proposition~\ref{prop:uM7}. When dealing  with Gibbs $u$-states, one considers disintegration along unstable manifolds (which are defined at every point, are tangent to $E^u$ and have uniform size). On the other hand, when one deals with Gibbs $cu$-states, one considers disintgration along Pesin unstable manifolds (which are defined almost everywhere, are tangent to $E^c \oplus E^u$, and do not have uniform size). Theorem~\ref{cu is closed} is new and Sections \ref{uniform estimates} and \ref{proof of prop} are entirely dedicated to the proof of this result.

%%%%%%%%%%%%%%%%%%%%%%%%%%%%%%%%%%%%%%%%%%%%%%%%%%%%%%%%%%%%%%%%%%%%%%%%%%%%%

\section{Uniform estimates of non-uniform hyperbolicity}\label{uniform estimates}

The apparently paradoxical title of this section reflects much of the spirit of non-uniform hyperbolicity in the presence of dominated splittings and partial hyperbolicity. Unlike 'genuine' non-uniformly hyperbolic systems, in which the angle between stable and unstable bundles may be arbitrarily small, these often allow some explicit form of robustness. An important manifestation of such robustness properties is that the measure of sets on which certain degrees of hyperbolicity hold may be uniformly bounded away from zero or even uniformly close to one.  Central to this theme is our Pliss-like Lemma~\ref{Pliss-like lemma} which we shall now prove.

\subsection{Proof of Lemma~\ref{Pliss-like lemma}}

Just as in Ma\~n\'e's proof of Pliss' Lemma \cite[Lemma 11.8]{M1987}, we define a function $S:\{0, \ldots,N\} \rightarrow \R$ by taking $S(0)=0$ and $S(n)=\sum_{j =1}^n a_j-n \gamma$ for $1\leq n \leq N$. Defining $1<n_1<\dots<n_k \leq N$ as the maximal sequence such that
$S(n_i)\geq S(n)$ holds for every $0\leq n< n_i$ and $i =1, \ldots, k$, one may easily check that the $n_i$ satisfy (\ref{pliss condition}). It remains is to show is that $k \geq \theta N$.

We set $F=\{i\in \{1, \dots , N\}: a_i < \Gamma \}$ and write $\{1,\dots,N\}\setminus\{n_{1},\dots,n_{k}\}$ as the finite union $\bigcup_{\alpha\in\Lambda} I_\alpha$ of pairwise disjoint intervals in $\mathbb{N}$ with cardinality $|I_\alpha|$. Note that 
\begin{equation}\label{sum on interval}
\sum_{i \in I_\alpha} a_i < | I_\alpha| \gamma
\end{equation}
for every $\alpha \in \Lambda$, for else the maximality of the sequence $n_i$ would be violated. We can bound $a_i$ from below by either $L$ or $\Gamma$, depending on whether or not $i$ belongs to $F$. Therefore
\begin{equation}\label{split sum on interval}
\sum_{i \in I_\alpha} a_i  = \sum_{i\in I_\alpha \cap F} a_i + \sum_{i \in I_\alpha \cap F^c} a_i 
\geq |I_\alpha \cap F| L + |I_\alpha \cap F^c| \Gamma. 
\end{equation}
Combining (\ref{sum on interval}) and (\ref{split sum on interval}) we obtain
\begin{equation}\label{ineq 1}
|I_\alpha \cap F| L + |I_\alpha \cap F^c| \Gamma < |I_\alpha| \gamma.
\end{equation}
Using the identity $|I_\alpha| = |I_\alpha \cap F | + |I_\alpha \cap F^c |$, rearranging terms, and summing over $\alpha$, (\ref{ineq 1}) becomes
\begin{equation}\label{ineq 2}
(\Gamma-L) \sum_{\alpha \in \Lambda} |I_\alpha \cap F|> (\Gamma-\gamma) \sum_{\alpha \in \Lambda} |I_\alpha|.
\end{equation}

Recall that $\{ I_\alpha: \alpha \in \Lambda \}$ is the family of disjoint intervals in $\{1,\dots,N\}\setminus\{n_{1},\dots,n_{k}\}$. In particular,
\begin{equation}\label{sum of intervals}
\sum_{\alpha \in \Lambda} |I_\alpha| = N-k.
\end{equation}
Moreover,  by hypotheses we have $ |F| \leq \kappa N$, and then
\begin{equation}\label{cardinality of F}
\sum_{\alpha \in \Lambda} |I_\alpha \cap F| \leq |F| \leq \kappa N.
\end{equation}
Combining (\ref{ineq 2}) with (\ref{sum of intervals}) and (\ref{cardinality of F}) gives
\begin{eqnarray}
(\Gamma - L) \kappa N &>& (\Gamma-\gamma) (N-k)\nonumber\\
&>& (\Gamma-\gamma)N -(\Gamma-\gamma) k\label{ineq 3}
\end{eqnarray}
Rearranging terms in (\ref{ineq 3}) we obtain
$$(\Gamma -\gamma) k>\left[(\Gamma -\gamma) -(\Gamma - L) \kappa\right] N$$
and then
$$
k> \left[1-\frac{\Gamma - L}{\Gamma -\gamma}\kappa\right] N$$
\noindent shows that $k > N \theta$, where $\theta=\left[1-\frac{\Gamma - L}{\Gamma -\gamma}\kappa\right]$.
%\end{proof}

\subsection{Abundance of hyperbolic times}

We recall (see \cite{ABV2000}) that, given $0<\sigma<1$,  an integer $n\geq 0$ is a {\em $\sigma$-hyperbolic time} for $x\in M$ if 
\begin{equation*}
\prod_{j=n-k+1}^n \| Df^{-1} \vert E_{f^j(x)}^{cu}\| \leq \sigma^k \quad \text{ for all } 1\leq k \leq n.
\end{equation*}

Recall from section \ref{measure section} that  if  $f\in \ME$ and  $\mu\in \Gu(f)$ then
$$\lambda^c(f,x)=\liminf_{n \to \infty} \frac{1}{n} \log \| (Df^n \vert E_x^{cu} )^{-1}\|^{-1} >0$$ 
$\mu$-almost everywhere. The integrated minimum central Lyapunov exponent $\hat{\lambda}^c(f,\cdot)$ is therefore positive on $\Gu(f)$. Since $\Gu(f)$ is compact (Proposition~\ref{prop:uM7})  and $\hat{\lambda}^c$ is lower semi-continuous on $\Gu(f)$ (Proposition~\ref{semicont of Lyapunov exponent}), there is a positive lower bound for $\hat{\lambda}^c(f,\mu)$ on $\Gu(f)$.
We can therefore  fix some $0<\sigma<1$ such that
\begin{equation}\label{sigma}
0< \log \sigma^{-1} < \inf_{\mu \in \Gu(f)} \hat{\la}^c(f,\mu)
\end{equation} 
and write
\begin{equation*}
\tau_x^\ell(f,\sigma) = \{ n \in \mathbb{N}: n \text{ is a $\sigma^\ell$ hyperbolic time for $x$ under $f^\ell$} \}.
\end{equation*}
When it is not necessary to emphasize the dependency of $\sigma$, by simplicity we  write $\tau_x^\ell(f,\sigma) = \tau_x^\ell(f)$. 

The next result says that the frequency of hyperbolic times can be taken arbitrarily close to one. The price to pay is that we may have to take a large iterate of $f$.

First we fix some notation. In what follows, we  use bold capital letters $\mathbf{U}, \mathbf{V},...$ to denote open sets  of the (fibered) space $\Gu(\mathcal{U}_{\mathcal{ME}})$ in the topology induced by the product topology on $\textrm{Diff}^r(M)\times \mathbb{M}^1(M)$.

\begin{prop}\label{density of hyperbolic times}
Let $f:M \rightarrow M$, $r>1$, be a $C^r$ mostly expanding diffeomorphism, $\mu$ a Gibbs $u$-state of $f$, and let $0< \sigma< 1$ be  such that (\ref{sigma}) holds. Then, given any $\epsilon>0$,  there exists a neighborhood $\mathbf{U}$ of $(f,\mu)$ in $\Gu(\mathcal{U}_{\mathcal{ME}})$ and some natural number $\ell_0\geq 1$ such that for every $(g,\nu) \in \mathbf{U}$, and every $\ell \geq \ell_0$, there is a set $A \subset M$ with $\nu(A)>1-\epsilon$ such that
\begin{equation*}
\liminf_{N \rightarrow \infty} \frac{ | \tau_x^\ell (g, \sigma) \cap \{1, \ldots, N\}|}{N}  \geq 1-\epsilon  
\end{equation*} 
for every $x \in A$.

\end{prop}

Before proving Proposition~\ref{density of hyperbolic times} we need an auxiliary result. There is a  well known characterization of weak* convergence of probability measures on a compact metric space, saying  that a sequence of measures $\mu_n$ converges to $\mu$ if and only if $\liminf_{n \rightarrow \infty} \mu_n (U)\geq \mu(U)$ whenever $U$ is an open set. In other words, the function
\begin{equation*}
\M^1(M) \ni \mu \mapsto \mu(U) \in \mathbb{R}
\end{equation*}
is lower semi-continuous whenever $U \subset M$ is open. Lemma~\ref{open in product topology} can be seen as a slight variation of that. 

Let $C^0(M, \mathbb{R})$ be the space of continuous functions $M \to \mathbb{R}$ endowed with the $C^0$ topology.

\begin{lemma} \label{open in product topology}
For any $\varphi \in C^0(M, \mathbb{R})$ denote by $U_{\varphi}\subseteq M$ the (open) set on which $\varphi$ is positive. Then the map
\begin{equation*}
C^0(M, \mathbb{R}) \times \M^1(M) \ni (\varphi,\mu) \mapsto \mu(U_{\varphi}) \in \mathbb{R} 
\end{equation*}
is lower semi-continuous in the product topology on $C^0(M, \mathbb{R})\times \M^1(M)$.
\end{lemma}

The proof is straightforward but included for the sake of completeness.

\begin{proof}
Fix some pair $(\varphi,\mu) \in C^0(M, \mathbb{R})\times \M^1(M)$ and an arbitrary $\epsilon>0$. We need to show that there are  open neighborhood $\mathcal{U}\subseteq C^0(M, \mathbb{R})$ of $\varphi$ and $\mathbb{U}\subseteq  \M^1(M)$ of $\mu$  such that $\nu(U_{\phi})>\mu(U_{\varphi})-\epsilon$ for every $(\phi,\nu) \in \mathcal{U}\times \mathbb{U}$.

By regularity of $\mu$ there is some compact set $C\subset U_{\varphi}$ such that $\mu(C)>\mu(U_{\varphi})-\epsilon$. Since $\varphi$ is positive on $C$, it follows by compactness that we can find some number $\beta$ that satisfies $0<\beta<\inf_{x \in C} \varphi(x)$. Observe that $U_{\varphi-\beta} \supset C$.

Let $\mathcal{U}$ be the open ball of radius $\beta$ around $\varphi$ in $C^0(M, \mathbb{R})$. Thus if $\phi \in \mathcal{U}$ and $x \in U_{\varphi-\beta}$ we have $\phi(x)> \varphi(x)- \beta>0$. Hence
$$
C\subset U_{\varphi-\beta} \subset U_{\phi} 
$$
for every $\phi \in \mathcal{U}$.

Let $\rho:M \rightarrow [0,1]$ be a continuous function satisfying $ \rho \vert C = 1$  and $ \rho \vert U_{\varphi-\beta}^c  = 0.$
In particular, for every $\phi \in \mathcal{U}$, 

 $$\mu(U_{\phi})\geq \mu(U_{\varphi-\beta})\geq \int \rho \ d\mu \geq \mu(C)> \mu(U_{\varphi})-\epsilon$$

Now let $\mathbb{U}$ be the open neighborhood of $\mu$ in $\M^1(M)$ defined by
\begin{equation*}
\mathbb{U}=\{\nu \in \M^1(M): \int \rho \ d\nu> \mu(U_{\varphi})-\epsilon \}.
\end{equation*}
Then, if $(\phi,\nu) \in \mathcal{U} \times \mathbb{U}$ we have
\begin{equation*}
\nu(U_{\phi})\geq \nu(U_{\varphi-\beta})\geq \int \rho \ d\nu \geq \nu(C)> \nu(U_{\varphi})-\epsilon. 
\end{equation*}
\end{proof}

\begin{proof}[Proof of Proposition~\ref{density of hyperbolic times}] 

Recall the choice of $0<\sigma<1$ in \eqref{sigma}. We write $\gamma= \log \sigma^{-1}$ and fix some $\Gamma$ with 
\begin{equation}\label{eq:Gamma}
\gamma < \Gamma <  \inf_{\mu \in \Gu(f)} \hat{\la}^c(f,\mu). 
\end{equation} 
We also fix some 
$$L < \inf_{x \in M} \log \|Df^{-1}\vert E_x^{cu}\|^{-1}.$$
For every $\ell \geq 1$ and $g\in \ME$, we define the family $\{U_g^\ell\} $ of opens sets in $M$ by
\begin{equation*}
U_g^\ell = \{ x \in M: \frac{1}{\ell} \log \|(Dg^\ell \vert E_x^{cu})^{-1}\|^{-1} > \Gamma \}\end{equation*}

For every $\mu\in\Gu(f)$, from \eqref{eq:Gamma} we have
$$\Gamma <  \int\lambda^c(f,x)d\mu=\int \liminf  \frac{1}{\ell}\log\|(Df^\ell|E^{cu})^{-1}\|^{-1}d\mu$$
Since $\lambda^c(f,\cdot)$ is $f$-invariant, if $\mu$ is ergodic, then for $\mu$-almost every $x\in M$
\begin{equation}\label{eq:Gamma2}
\lambda^c(f,x)=\lim\inf \frac{1}{\ell}\log\|(Df^\ell|E^{cu}_x)^{-1}\|^{-1}>\Gamma.
\end{equation}
If $\mu$ is not ergodic, then we write $\mu$ as convex combination of ergodic measures (see Proposition~\ref{prop:uM7}), all of them are Gibbs $u$-state satisfiying \eqref{eq:Gamma2}. So any convergent subsequence of   $\frac{1}{\ell}\log\|(Df^\ell \vert E_x^{cu})^{-1}\|^{-1}$ must  converges $\mu$-almost everywhere to some limit larger than $\Gamma$.  Hence, we must have 
\begin{equation*}
\lim_{\ell \rightarrow \infty} \mu(U_f^\ell) = 1.
\end{equation*} 

Fix $\epsilon>0$. Take 
\begin{equation}\label{eq:kappa}
\kappa = \epsilon \frac{\Gamma-\gamma}{\Gamma-L}
\end{equation} 
and choose $\ell_0\geq 1$ so that  for every $\ell\geq\ell_0$ we have
\begin{equation*}
\mu(U_f^\ell)>1-\epsilon \kappa.
\end{equation*}

It follows from Lemma~\ref{open in product topology} that the set
\begin{equation*} 
\mathbf{U} = \{(g,\nu): \nu(U_g^\ell)>1-\epsilon \kappa  \text{ and } 
\inf_{x \in M} \log \|Dg^{-1}\vert E_x^{cu}\|^{-1} > L \}
\end{equation*}
is open in $\Gu(\mathcal{U}_{\mathcal{ME}})$.

Pick any pair $(g,\nu) \in \mathbf{U}$. We shall prove that $(g,\nu)$ satisfies the conclusion of Proposition~\ref{density of hyperbolic times}.  

Consider the function 
\begin{equation*} 
F(x) = \lim_{n \rightarrow \infty} \frac{1}{n} \#\{0 \leq k \leq n-1: g^{\ell k}(x) \in U_g^\ell \}
\end{equation*}
of the frequency of visits to the set $U_g^\ell$. By Birkhoff's Ergodic Theorem  it is well defined $\nu$-almost everywhere and satisfies 
\begin{equation}\label{eq:desF}
\int F \ d\nu = \nu(U_g^\ell )>1-\epsilon \kappa.
\end{equation}

Let $A = \{x \in M: F(x)>1-\kappa\}$. Chebyshev's inequality  and \eqref{eq:desF} gives 
$$\nu(M \setminus A) = \nu(\{x: 1-F(x) \geq \kappa \}) \leq \frac{1}{\kappa} \int 1-F \ d\nu
< \frac{\epsilon \kappa}{\kappa} = \epsilon. $$
In other words, $\nu(A) > 1 - \epsilon$ and the proof will be complete once we have proved that for every $x \in A$,
\begin{equation*}
\liminf_{N \rightarrow \infty} \frac{ | \tau_x^\ell (g) \cap \{1, \ldots, N\}|}{N}  \geq 1-\epsilon .
\end{equation*} 
To this end, fix $x\in A$ and  $N_0\geq 1$ is such that 
\begin{equation}\label{high frequency of visits}
\frac{1}{N} \#\{0 \leq k \leq N-1: g^{\ell k}(x) \in U_g^\ell \} >1-\kappa
\end{equation} 
for every $N \geq N_0$. Let
$$a_i = \frac{1}{\ell} \log \|(Dg^\ell \vert E_{g^{\ell(i-1)}(x)}^{cu})^{-1}\|^{-1}.$$ 
Then (\ref{high frequency of visits}) implies that
\begin{equation*}
\# \{i \in \{1, \ldots, N \} : a_i < \Gamma \} \leq \# \{i \in \{1, \ldots, N \} : a_i \leq \Gamma \} <\kappa N.
\end{equation*} 
We can therefore conclude from Lemma~\ref{Pliss-like lemma} that, for $\epsilon>0$, there exist $\kappa>0$ defined by \eqref{eq:kappa} and there exist $1<n_1< \ldots < n_{m} \leq N$ with $m > (1-\kappa \frac{\Gamma-L}{\Gamma-\gamma}) N = (1-\epsilon)N$ such that
\begin{equation}
\sum_{j=n+1}^{n_i} \frac{1}{\ell} \log \|(Dg^\ell \vert E_{g^{\ell(j-1)}(x)}^{cu})^{-1}\|^{-1} 
=  \frac{1}{\ell} \log  \prod_{j=n+1}^{n_i}\| Dg^{-\ell} \vert E_{g^{\ell j}(x)}^{cu}\|^{-1}  \geq \gamma (n_i-n). \label{hyp time 1}
\end{equation}
for every  $1\leq i \leq m$ and every $0\leq n <n_i$. Writing $k=n_i-n$ and remembering that $\gamma = \log \sigma^{-1}$, (\ref{hyp time 1}) may be more conveniently expressed by
\begin{equation*}
\prod_{j=n_i-k+1}^{n_i} \|Dg^{-\ell}\vert E_{g^{\ell j}(x)}^{cu} \| \leq \sigma^{\ell k}
\end{equation*}
for every $1\leq i \leq m$ and every $1 \leq k\leq n_i$. That is, each $n_i$ is a $\sigma^\ell$ hyperbolic time for $x$ under $g^\ell$.
\end{proof}

\subsection{Pesin blocks of uniform measure}

We now change our focus a bit. Fix $0<\sigma<1$. Instead of considering hyperbolic times of a given point $x$, we consider the set  
\begin{equation}\label{eq:deflambdan}
\Lambda_\ell^n(f,\sigma) = \{x\in M\::\:  \prod_{j=0}^{k-1} \|Df^{-\ell}\vert E_{f^{-\ell j}(x)}^{cu}\| \leq \sigma^{\ell k}, \, \forall 1 \leq k \leq n\}
\end{equation}
of points which are hyperbolic time iterates of some other point. We are particularly interested in the set 
\begin{equation}\label{eq:deflambda}
\Lambda_\ell (f,\sigma) = \bigcap_{n \geq 1} \Lambda_\ell^n(f,\sigma),
\end{equation}
which we call a \emph{Pesin like block} of $f$. When it is not necessary to emphasize the dependency of $\sigma$,  we  write $\Lambda_\ell^n(f)=\Lambda_\ell^n(f,\sigma)$ or $\Lambda_\ell (f)=\Lambda_\ell (f,\sigma)$ respectively in order to simplify notation. 

\begin{remark}
The Pesin like blocks $\Lambda_\ell (f)$ are different from the Pesin blocks $\Bl(\ell, f^{-1})$ considered by Avila and Bochi  in \cite{AB2012}. For example, for points in $\Lambda_\ell(f)$, the Lyapunov exponent in the $E^{cu}$ bundle is bounded below by a fixed number $\log \sigma^{-1}$, whereas for points in $\Bl(\ell, f^{-1})$, they are bounded below by $1/\ell$. Our notion is therefore more restrictive, and suitable to a situation where Lyapunov exponents are almost everywhere bounded away from zero with respect to a relevant set of measures (which is not the case in \cite{AB2012}). A main ingredient in our work is that $\Lambda_\ell(f)$ has large $\mu$-measure for large $\ell$  and $\mu \in \Gcu(f)$ in a way which is uniform in a neighborhood of $f$ (see Lemma~\ref{large pesin blocks}). It is for this reason that we have proved the Pliss-like Lemma (Lemma~\ref{Pliss-like lemma}). Avila and Bochi obtain similar results for the set $\Bl(\ell, f^{-1})$ using a very elegant application of the Maximal Ergodic Theorem . The current work could perhaps be made a few pages shorter by working with $\Bl(\ell, f^{-1})$ rather than $\Lambda_\ell(f)$ and making use of their results. However, we think that our estimates on the size of $\Lambda_\ell(f)$ is of independent interest, as well as being more intuitive for those who are used to arguments involving Pliss' Lemma.
\end{remark}

\begin{lemma}\label{large hyperbolic set of iterate}
Given $f:M \rightarrow M$ mostly expanding, $\mu \in \Gu(f)$, a number $0<\sigma<1$ satisfying (\ref{sigma}) and $\epsilon>0$, there exist a neighborhood $\mathbf{U}$ of $(f,\mu)$ in $\Gu(\mathcal{U}_{\mathcal{ME}})$ and an integer $\ell_0\geq 1$ such that $\nu(\Lambda_\ell (g, \sigma)) > 1-\epsilon$ for every $(g,\nu) \in \mathbf{U}$ and every $\ell \geq \ell_0$.
\end{lemma} 

\begin{proof}
Fix $(f,\mu) \in \Gu(\mathcal{U}_{\mathcal{ME}})$ and $\epsilon>0$ arbitrarily. Choose some $\epsilon'>0$ small enough that 
\begin{equation}\label{eq:epsilon}
(1-\epsilon')^2>1-\epsilon.
\end{equation}

Proposition~\ref{density of hyperbolic times} guarantees the existence of an open neighborhood $\mathbf{U}$ of $(f,\mu)$ in $\Gu(\mathcal{U}_{\mathcal{ME}})$ and a positive integer $\ell_0\geq1$ such that, given any $(g,\nu) \in \mathbf{U}$, there is some set $A \subset M$, with $\nu(A)>1-\epsilon'$, such that 
\begin{equation*}
\liminf_{N \rightarrow \infty} \frac{ | \tau_x^\ell (g) \cap \{1, \ldots, N\}|}{N}  > 1-\epsilon'  
\end{equation*} 
for every $x \in A$. We will prove that if $(g,\nu)$ belongs to $\mathbf{U}$, then $\nu(\Lambda_\ell^n(g))>1-\epsilon$. 
Let
\begin{equation*}
A_n=\{x \in M: \inf_{k \geq n}  \frac{ | \tau_x^\ell (g) \cap \{1, \ldots, k\}|}{k}>1-\epsilon'\}.
\end{equation*}
Note that $A_n$ is an increasing sequence of measurable sets  such that $A\subset \cup A_n$ and, by our choice of $\mathbf{U}$ and $\ell$, we have that $\nu(\bigcup_{n \in \mathbb{N}} A_n)>1-\epsilon'$. Therefore, we can (and do) fix some integer $N\geq 1$ such that $\nu(A_N)>1-\epsilon'$. Likewise, let
\begin{equation*}
B_n = \{ x \in M: \frac{ | \tau_x^\ell (g) \cap \{1, \ldots, n\}|}{n}>1-\epsilon'\}.
\end{equation*}
The sequence $B_n$ does not have to be increasing, but we have $B_n \supset A_n$ for every $n \in \mathbb{N}$ so that, in particular, 

\begin{equation}\label{eq:measureBn}
\nu(B_N)> 1-\epsilon'.
\end{equation}

Denotes by $\chi_B$ the characteristic function of a Borelean set $B\subseteq M$. Observe that
\begin{equation*}
\sum_{n=1}^N \chi_{\Lambda_\ell^n(g)}\circ g^{n\ell} (x) = |\tau_x^\ell \cap \{1, \ldots, N\}|
\end{equation*}
for every $x \in M$. Consequently

\begin{eqnarray*} 
\nu(\Lambda_\ell^n(g))  &=& \int \chi_{\Lambda_\ell^n(g)} \ d\nu \\
&=& \int \frac{1}{N} \sum_{n=1}^N \chi_{\Lambda_\ell^n(g)}\circ g^{n\ell} \ d\nu \\
&\geq& \int_{B_N}  \frac{1}{N} \sum_{n=1}^N  \chi_{\Lambda_\ell^n(g)}\circ g^{n\ell} \ d\nu \\
&\geq& \int_{B_N} 1-\epsilon' \ d\nu.
\end{eqnarray*}
Recalling  \eqref{eq:measureBn} and \eqref{eq:epsilon}, we obtain 

$$\nu(\Lambda_\ell^n(g)) >\int_{B_N} 1-\epsilon' \ d\nu>(1-\epsilon')^2>1-\epsilon.$$

Recall that $\Lambda_\ell(g) = \bigcap_n \Lambda_\ell^n (g)$, and that the $\Lambda_\ell^n(g)$ form a nested decreasing sequence in $n$. The proof follows readily. 
\end{proof}

%\begin{remark}\label{rmk:singlemu} Let $f:M\to M$ be a $C^r$-partially hyperbolic diffeomorphism, $r>1$, and let $\mu$ be a Gibbs $u$-state for $f$ such that for $\mu$-almost every point $x\in M$, $\lambda^c(f,x)>0$. Let  $\hat{\lambda}^c(f,\mu):=\int\lambda^c(f,x)d\mu>0$, and fix any $0<\sigma<1$  such that $\log \sigma^{-1}<\hat{\lambda}^c(f,\mu)$. Then the pertinent conclusions of Proposition~\ref{density of hyperbolic times} and Lemma~\ref{large hyperbolic set of iterate} remain valid for the system $(f,\mu)$. In particular, given  $\epsilon>0$, there exist an integer $\ell_0:=\ell_0(f,\mu, \sigma)\geq 1$ such that $\mu(\Lambda_\ell (f,\sigma)) > 1-\epsilon$ for every  $\ell \geq \ell_0$.  Nevertheless uniformity of $\ell_0$ with respect of  $f$, $\mu$ and $\sigma$  follows from the mostly expanding condition and compactness of Gibbs $u$-states as we see in the next lemma.
%\end{remark}

Next Lemma is an improvement of Lemma \ref{large hyperbolic set of iterate}. It says that, not only can  the number $\ell_0$ in Lemma \ref{large hyperbolic set of iterate} be taken uniform in a neighbourhood of the pair $(f,\mu)$ in $\Gu(\mathcal{U}_{\mathcal{ME}})$; it can indeed be choosen so that it holds simultaneously for every Gibbs $u$-state of every diffeomophism in a neighbourhood of $f$.

\begin{lemma}\label{large pesin blocks}
Let $f:M \to M$ be a $C^r$  mostly expanding diffeomorphism, $r>1$. Given any $\epsilon>0$, there exists $\ell_0$ and a $C^r$ neighborhood $\mathcal{U}$ of $f$ such that $\nu(\Lambda_\ell(g))>1-\epsilon$ for every $\ell \geq \ell_0$, $g \in \mathcal{U}$ and $\nu \in \Gu(g)$.
\end{lemma}

\begin{proof}
Fix $\epsilon>0$. 
 Since $\Gu(f)$ is compact (see Proposition~\ref{prop:uM7}), it follows from Lemma~\ref{large hyperbolic set of iterate} that there are open sets $\mathbf{U}_1,\dots,\mathbf{U}_n\subset \Gu(\mathcal{U}_{\mathcal{ME}})$, where $\mathbf{U}_i=\mathcal{U}_i \times \mathbb{U}_i$,  and integers $\ell_1, \ldots, \ell_n$ such that 
\begin{enumerate}
\item[(i)] $\Gu(f) \subset  \mathbb{U}_1 \cup \ldots \cup  \mathbb{U}_n$, and
\item[(ii)] $\nu(\Lambda_{\ell} (g))>1-\epsilon$ whenever $(g,\nu) \in \mathbf{U}_i$ for some $i=1, \ldots, n$ and every $\ell\geq \ell_i$.
\end{enumerate} 

Let $\mathcal{U} = \mathcal{U}_1 \cap \ldots \cap \mathcal{U}_n$ and $\mathbb{U} = \mathbb{U}_1 \cup \ldots \cup \mathbb{U}_n$. Then
\begin{equation*}
 \Gu(f) \subset \mathcal{U} \times \mathbb{U} \subset \bigcup_{i=1}^{n} \mathcal{U}_i \times \mathbb{U}_i.
 \end{equation*}
 It follows from Proposition~\ref{prop:uM7} that, upon possibly reducing $\mathcal{U}$, we may (and do) suppose that $\Gu(g) \subset \mathcal{U} \times \mathbb{U}$ for every $g \in \mathcal{U}$. Let $\ell_0= \prod_{i=1}^n \ell_i $.  Given any $g \in \mathcal{U}$ and any $\nu \in \Gu(g)$ there exists some $i = 1, \ldots, n$ such that $(g,\nu) \in \mathcal{U}_i \times \mathbb{U}_i$. Since $\ell_0\geq \ell_i$, for every $i=1,\dots, n$, from  (ii) above we have  $\nu(\Lambda_\ell(g)) >1-\epsilon$ for every $\ell\geq \ell_0$. 
 
\end{proof}

\subsection{Unstable manifolds and uniform densities}

Central to our argument is that the size of local unstable manifolds can be controlled on the sets $\Lambda_\ell(f,\sigma)$, uniformly in a neighbourhood of a given mostly expanding diffeomorphism. This was done from scratch in \cite[Theorem 4,7]{AB2012} using graph transforms. In this work some further properties of unstable manifolds are needed which are not stated in \cite[Theorem 4,7]{AB2012}.

\begin{Theorem}\label{unstable manifolds}
Let $f: M \to M$ be a $C^r$ mostly expanding diffeomorphism, with $r>1$. Then, given any $\ell \in \N$, there are a $C^r$ neighborhood $\cU$ of $f$ and  numbers  $r=r(\ell)>0, C=C(\ell)\geq 0, \delta=\delta(\ell)>0$ for which the following holds:
\begin{enumerate}
\item Given any $g \in \cU$  and $x \in \Lambda_\ell(g, \sigma)$, there is a $C^1$ embedded disk $W_{r}^{cu}(g,x)$ of dimension $\dim E^{cu}$  and radius ${r}>0$, centered at $x$, such that 
\[T_y W_{r}^{cu}(g,x) =E_y^{cu}(g)\]
 for every $y \in W_{r}^{cu}(g,x)$. 
 
\item $W_{r}^{cu}(g,x)$ depends continuously on both $x$ and $g$  in the $C^1$ topology;

\item $W_{r}^{cu}(g,x) \subset \Lambda_\ell(g,\sigma^{1/2})$;

\item if $y \in W_{r}^{cu}(g,x)$, then 
\[\di(f^{-n}(x),f^{-n}(y)) \leq C\sigma^{n/2}\di(x,y)\]
for every $n\geq 0$;

\item if  $y,z  \in \Lambda_\ell (g, \sigma) \cap B_{\delta} (x)$, then either 
\[W_{r}^{cu}(g,y)\cap W_{r}^{cu}(g, z) = \emptyset\]
 or  
\[W_{r}^{cu}(g, y) \cap B_{2 \delta}(x) = W_{r}^{cu}(g, z)\cap B_{2 \delta}(x).\]
\end{enumerate}
\end{Theorem}

We give a complete proof of Theorem \ref{unstable manifolds}. Our approach is a combination of the methods of \cite{ABV2000} and \cite{Abdenur2011}. In particular, we make use of so-called locally invariant plaque families introduced in \cite[Theorem 5.5]{HPS1977}. Unfortunately, \cite{HPS1977} only considers plaque families of a single diffeomorphism, but it is implicit in the construction that these plaques depend continuously, not only on the point in the manifold, but also as a function of the diffeomorphism in question. We cite a version of the plaque family theorem due to \cite[Lemma 3.5]{CP2015} that takes this into account.

\begin{Theorem}[\cite{HPS1977,CP2015}]\label{plaques}
Suppose that a diffeomorphism $f: M \to M$ has a dominated splitting $F \prec G$. Then there exist a $C^1$ neighborhood $\cU$ of $f$, a number $\rho>0$, and a continuous family of embeddings
\[\cU \times M \ni (g,x) \mapsto \Phi_{g,x} \in \Emb^1(\R^{\dim(F)}, M) \]
such that for every $g \in \cU$ and every $x \in M$ we have 
\begin{itemize}
\item $D \Phi_{g,x} (0) = F_x (g) \text{, and}$
\item $g ( \Phi_{g,x}(B_{\rho}(0))) \subset \Phi_{g,g(x)} (\R^{\dim F})$.
\end{itemize}
\end{Theorem}

We make some remarks and set some notation that will be used in the proof of Theorem~\ref{unstable manifolds}. Fix a mostly expanding diffeomorphism $f: M \to M$. The inverse of $f$ has a dominated splitting $E^{cu} \prec E^s$. Hence, according to Theorem~\ref{plaques} there are a $C^1$ neighborhood $\cU$ of $f$, a number $\rho>0$ and a family $\{\Phi_{g,x}\}_{g \in \cU, x \in M}$ such that $D \Phi_{g,x}(0) = E_x^{cu}(g)$ and 
\[g^{-1}(\Phi_{g,x}(B_\rho(0))) \subset \Phi_{g,g^{-1}(x)}(\R^{cu}),\]
where $cu$ stands for the dimension of $E^{cu}$.

 Let 
\[V_g(x) = \Phi_{g,x}(\R^{cu})\]
and for $c>0$ let 
\[ D_g(x,c) = \{y \in V_g(x): \di_{V_g(x)}(x,y) < c \}\]
and write $\partial D_g(x,c)$ for $\overline{D_g(x,c)}\setminus D_g(x,c)$. 
For sufficiently small $c$ (but uniform in $x$ and $g$), the set $D_g(x,c)$ is a disk of radius $c$ in the sense that
\[\di_{V_g(x)}(x,y) = c\]
for every $y \in \partial D_g(x,c)$. 
Take $c_0>0$ small enough so that $D_g(x,c_0)$ is a disk of radius $c_0$ and also so that 
\[D_g(x,c_0) \subset \Phi_{g,x}(B_\rho(0))\]
for every $(g,x) \in \cU \times M$. Then $g^{-1}(D_g(x,c_0)) \subset V_g(g^{-1}(x))$ for every pair $(g,x) \in \cU \times M$. The same is true if $c_0$ is replaced with any number smaller than $c_0$. In particular, given any $c_1 \leq c_0$, there exists $c_2 < c_1$ such that 
\[ g^{-1}(D_g(x,c_2)) \subset D_g(g^{-1}(x), c_1)\]
for every pair $(g,x) \in \cU \times M$. This can be carried out for iterates of $g$ as well: Given any $\ell \in \N$ and any $c_1<c_0$ there exists $c_2<c_1$ such that 
\[g^{-k}(D_g(x,c_2)) \subset D_g(g^{-k}(x),c_1)\]
for every pair $(g,x) \in \cU \times M$ and every $1 \leq k \leq \ell$.

\begin{proof}[Proof of theorem \ref{unstable manifolds}]

Fix a mostly expanding diffeomorphism $f: M \to M$, $0<\sigma<1$ as in (\ref{sigma}), and some number $\ell \in \N$. Let $\cU$, $\{V_g(x)\}_{g \in \cU, x \in M}$ and $c_0$ be as in the above discussion. We shall also assume that $c_0$ is sufficiently small so that each $D_g(x,c_0)$ is uniformly transversal to $E^s(g)$ for every $(g,x) \in \cU \times M$. By uniform continuity of $\log \| Df^{-\ell} \vert E_x^{cu}\|$, there exists $0<c_1< c_0$  such that 
\[\|Df^{-\ell} \vert E_y^{cu}(f)\| < \sigma^{-1/2} \|Df^{-\ell} \vert  E_x^{cu}(f)\| \]
whenever $\di(x,y) < c_1$. Upon possibly reducing $\cU$, we may assume that 
\[\|Dg^{-\ell} \vert E_y^{cu}(g)\| < \sigma^{-1/2} \|Dg^{-\ell} \vert  E_x^{cu}(g)\| \]
for every $g \in \cU$ and every pair $x,y \in M$ such that $\di(x,y) < c_1$.

Let $r=r(\ell) < c_1$ be such that 
\[g^{-i}(D_g (x,r)) \subset D_{g}(g^{-i}(x), c_1)\]
for every $g \in \cU$,  $x \in M$,  and $1 \leq i \leq \ell$.

Suppose in what follows that $g \in \cU$ and that $x \in \Lambda_\ell(g)$ for some $\ell$. We claim that for every $y \in D_{g}(x,r)$ and every $n \geq 1$ we have
\begin{equation}\label{ind1}
\prod_{k=0}^{n-1} \| Dg^{-\ell} \vert E_{g^{-\ell k}(y)}^{cu} (g) \| \leq \sigma^{\ell n/2},
\end{equation}

Let us prove this by strong induction. For $n=1$, (\ref{ind1}) follows directly from the definition of $\Lambda_\ell(g,\sigma)$. Suppose that (\ref{ind1}) holds for every $1 \leq n \leq N$. We shall prove that it also holds for $n=N+1$. 

First notice that, since  (\ref{ind1}) is assumed to hold for every $y \in D_g(x,r)$, it implies that 
\begin{equation}\label{ind2}
\di_{V_g (g^{-\ell n}(x))} (g^{-\ell n}(x), g^{-\ell n}(y)) \leq \sigma^{\ell n/2} \di_{V_g(x)} (x,y)
\end{equation}
for every $y \in D_g(x,r)$. Consequently, 
\begin{equation}\label{ind3}
g^{-\ell n} (D_{g}(x,r)) \subset D_{g}(g^{-\ell n}(x), r\, \sigma^{\ell n/2}) \subset D_{g}(g^{-\ell n}(x), c_1).
\end{equation}

Now,(\ref{ind3}) said in particular that for every $1\leq n \leq N$ and for every $y \in D_g(x,r)$ we have 
\[\di (g^{-\ell n}(x), g^{-\ell n}(y)) 
\leq r\, \sigma^{\ell n/2} < r < c_1.\]
Hence
\[\|Dg^{-\ell } \vert E_{f^{-\ell n}(y)}^{cu} \| \leq \sigma^{-1/2} \|Df^{-\ell } \vert E_{f^{-\ell n}(y)}^{cu} \|\]
for every $1 \leq n \leq N$.
It follows that
\[\prod_{k=0}^{N} \| Dg^{-\ell} \vert E_{g^{-\ell k}(y)}^{cu} \|
\leq \sigma^{-(N+1)/2} \prod_{k=0}^{N} \| Dg^{-\ell} \vert E_{g^{-\ell k}(x)}^{cu} \| 
\leq \sigma^{-(N+1)/2} \sigma^{\ell (N+1)} \leq \sigma^{\ell (N+1)/2},\]
which is just our inductive hypothesis for $n=N+1$. 
The induction is thus complete and we conclude that (\ref{ind1}) (and hence (\ref{ind2}) and (\ref{ind3})) holds for every $n \geq 1$.

Let $W_{r}^{cu}(g,x) = D_g(x, r)$. Then (ii) is true by construction. Item (iii) is given by our inductive step and (iv) is a consequence thereof. To see why (i) holds, we argue by contradiction. Suppose that for some $g \in \cU$ and $x \in \Lambda_\ell(g,\sigma)$ we have that $W_{r}^{cu}(x)$ is not tangent to $E^{cu}(g)$. Then there is some $y \in W_{r_\ell}^{cu}(x)$ and some $v \in T_y W_{r}^{cu}(x)$ such that $v \notin E_y^{cu}(g)$. By domination, the angle between $Dg^{-n}v$ and $E_{g^{-n}(y)}^s(g)$ then tends to zero as $n \to \infty$. But $Dg^{-n}v \in T D_g(g^{-n}(x), \delta_1)$ for every $n\geq 0$ and must therefore have angle to $E^s(g)$ which is bounded away from zero --- a contradiction. We conclude that indeed $W_{r}^{cu}(x)$ must be tangent to $E^{cu}(g)$.

It remains to prove item (v). To this end, we first show that for $g \in \cU$, the family $\{W_{r}^{cu}(x): x \in \Lambda_\ell (g,\sigma) \}$ is \emph{self-coherent}: given any $x,y \in \Lambda_\ell(g,\sigma)$, the intersection $W_{r}^{cu}(x) \cap W_{r}^{cu}(y)$ is an open subset of both $W_{r}^{cu}(x)$ and $W_{r}^{cu}(y)$. The argument is classic. Suppose it is not true. Then there is a point $z$ in $\partial (W_{r}^{cu}(x) \cap W_{r}^{cu}(y))$ which is in the interior of both $W_{r}^{cu}(x)$ and $W_{r}^{cu}(y)$. Let $w_1$ be a point in $W_{r}^{cu}(x)\setminus W_{r}^{cu}(y)$ close to $z$ and $w_2$ a point in $W_{r}^{cu}(y) \setminus W_{r}^{cu}(x)$ close to $z$ such that $w_1$ and $w_2$ lie on the same local stable manifold. By the construction of $W_{r}^{cu}(x)$ and $W_{r}^{cu}(y)$, 
\[\di(f^{-n}(w_1),f^{-n}(w_2)) \leq \di(f^{-n}(w_1),f^{-n}(z))+\di(f^{-n}(z),f^{-n}(w_2))
\leq 2 c_1 \]
for every $n \geq 0$. But $w_1$ and $w_2$ are on the same unstable manifold, so $\di(f^{-n}(w_1),f^{-n}((w_2))$ must grow larger than $2 c_1$.

Let $\delta=\delta(\ell)>0$ be small enough so that if $g \in \cU$,  $x \in M$, and $y\in B_{\delta}(x) \cap \Lambda_\ell(g,\sigma)$, then $\partial W_{r}^{cu}(y) \cap B_{2 \delta}(x) = \emptyset$. Now given any $y,z \in B_{\delta}(x) \cap \Lambda_\ell(g,\sigma)$, by the self-coherent property we have that 
\[\partial( W_{r}^{cu}(x) \cap W_{r}^{cu}(y)) \subset \partial W_{r}^{cu}(x) \cup \partial W_{r}^{cu}(y).\]
Hence 
\[W_{r}^{cu}(y) \cap B_{2 \delta}(x) = W_{r}^{cu}(z) \cap B_{2 \delta}(x)\]
as required.

\end{proof}

\subsection{Lamination bundles and disintegration}

The notion of Gibbs $cu$-states involves disintegration of a measure along Pesins' unstable manifolds. Broadly speaking, two ways to disintegrate a measure along unstable manifolds appear in the literature. One of them, used in \cite{ABV2000} and the works influenced by it, uses a so-called foliated box. In such a box, unstable manifolds are graphs of functions from one Euclidean ball to another Euclidean ball. This approach is often practical under the presence of dominated splittings. The other approach, often used in the more general setting of non-uniform hyperbolicity,  considers the union of unstable manifolds of points in the intersection of a Pesin  block with a small ball. Here we use a variation of this latter approach.

Let us set up some notation for this section. Let $f:M\to M$ be a $C^r$-partially hyperbolic diffeomorphism, $r>1$, and let $\mu$ be an ergodic Gibbs $u$-state with positive center Lyapunov exponents for $f$. Let  $\hat{\lambda}^c(f,\mu):=\int\lambda^c(f,x)d\mu>0$, and fix any $0<\sigma<1$  such that $\log \sigma^{-1}<\hat{\lambda}^c(f,\mu)$. It is well know from Pesin Theory, that $\mu$-a.e. $x\in M$ there exists $W^{cu}(x)$ center-unstable manifolds. We denote by 
$$\cW(f)=\{W\::\: W \mbox{ is a center-unstable manifold }\}.$$

For  every $\ell\geq 1$,  we fix  $r=r(\ell)>0$ and $\delta=\delta(\ell)>0$ as in Theorem~\ref{unstable manifolds}. This implies that there exist a center-unstable lamination
$$\cW_\ell(f,\sigma) = \{W_{r}^{cu}(x): x \in \Lambda_\ell(f,\sigma)\}.$$

By uniqueness of Pesin's unstable manifolds, we have that $W_r^{cu}(x) \subset W^{cu}(x)$ $\mu$-almost everywhere. In particular, $W_r^{cu}(x)$ $\mu$-almost everywhere of class $C^r$ rather than just $C^1$, as stated in Theorem~\ref{unstable manifolds}.

Given a mostly expanding diffeomorphism $f: M \to M$, we fix, for every $\ell \in \N$ values $r = r(\ell)$ and $\delta = \delta(\ell)$ so that the conclusions of Theorem \ref{unstable manifolds} hold in a neighborhood of $f$. Then, for every  $x\in M$, we write

\begin{equation} \label{partition}
\cQ(\ell,\sigma,x, f) = \{W_{r}^{cu}(y)\cap B_{2\delta} (x): y \in \Lambda_\ell(f,\sigma) \cap \overline{B_\delta (x)} \}
\end{equation}
and 
\begin{equation} \label{lamination bundle}
Q(\ell, \sigma, x,f) = \bigcup_{D \in \cQ(\ell,\sigma, x, f)} D.
\end{equation}

We refer to the set $Q( \ell, \sigma, x,f)$ as a \emph{lamination bundle}. Thus whenever we talk about ``the lamination bundle $Q( \ell, \sigma, x,f)$", the corresponding $r>0$ and $\delta>0$ are implicitly defined.  In particular (see  part v in Theorem \ref{unstable manifolds}) $\cQ(\ell,x)$ is a partition of $Q(\ell,x)$ for every $x\in M$. This is indeed the reason why, in the definition of lamination bundle, we take intersections of unstable manifolds with $B_{2 \delta}(x)$. Note also that we only consider unstable manifolds of points in the smaller ball $B_\delta(x)$. The reason for this is that by doing so we get a lower bound for the volume of the leaves in $\cQ(\ell, \sigma, x,f)$.  When it is not necessary to emphasize the dependency of one particular variable, $f$ or $\sigma$ for instance, for  simplicity of notation we will omit it,  writing for example $\cQ(\ell, x)=\cQ(\ell, \sigma,x, f) $ or $Q(\ell, x)=Q(\ell, \sigma, x, f)$ respectively.

Let $\mu$ be any Borel measure and  $Q = Q(\ell, x)$ a lamination bundle. Clearly, the partition  $\cQ = \cQ(\ell,  x)$ is  measurable in the sense of Rokhlin. We may therefore decompose $\mu_Q$ with respect to $\cQ$: there exists a measurable family of probability measures $\{\mu_D: D \in \cQ\}$ (usually called {\em conditional  measures}) and a measure $\hat{\mu}$ on $\cQ$ (usually called {\em factor measure}) with $|\hat{\mu}| = \mu(Q) $ such that 
\begin{equation}\label{disintegration}
\int_Q \varphi \ d\mu= \int \varphi \ d\mu_Q= \int_{\cQ} \left( \int_D \varphi(x) \ d\mu_D(x) \right) \ d\hat{\mu}(D)
\end{equation}
for every continuous $\varphi:M \to \mathbb{R}$.

Suppose that a  Gibbs $u$-state  $\mu$ with positive central Lyapunov exponent  is also a Gibbs $cu$-state for  $f$.   Fix $\epsilon>0$. By Lemma~\ref{large hyperbolic set of iterate}, there exists $\ell_0\geq 0$ such that $\mu(\Lambda_\ell(f))>1-\epsilon$ for every $\ell\geq \ell_0$. Now fix a such $\ell\geq \ell_0$ and consider  any $x\in M$ such that the lamination bundle $Q=Q(\ell,x)$ has  positive $\mu$-measure. Since $\mu$ is a $cu$-Gibbs state, and since Pesin's unstalbe manifold of $\mu$-almost every point $x$ in $Q$ coincides with $W_r^{cu}(x)$, the conditional measures $\mu_D$ are $\hat{\mu}_D$-almost everywhere absolutely continuous to $m_D$, i.e. there is a family of densities $\rho_D$ such that 
\begin{equation}\label{disintegration2}
\int_Q \varphi \ d\mu= \int \varphi \ d\mu_Q= \int_{\cQ} \left( \int_D \varphi(x)\rho_D(x) \ d m_D(x) \right),
\end{equation}
for every continuous function $\varphi:M \to \R$.   By the uniqueness of the disintegration,  comparing \eqref{disintegration} and \eqref{disintegration2} we conclude that $\hat{\mu}$-almost every $D\in\cQ$, $\mu_D=\rho_D \cdot m_D$, so that  $\mu_D$ is absolutely continuous with respect to $m_D$  for $\hat{\mu}$-almost every $D \in \cQ$. 

\begin{remark}\label{rmk:densidad} 
A priori, $\mu_D$ is only assumed to be absolutely continuous  with respect to $m_D$, so in principle the densities $\rho_D$ need only be  measurable and may be zero on sets of positive $m_D$-measure. Yet  it is known \emph{a fortiori} (see e.g. \cite[Theorem 13.1.2]{barreira_pesin_2007}) that in every lamination bundle above, and $\hat{\mu}$-almost every $D \in \cQ$, the density $\rho_D = \frac{d\nu_D}{dm_D}$ satisfies (compare with \eqref{eq:udensity})
\begin{equation}\label{density}
\frac{\rho_D(x)}{\rho_D(y)} = \prod_{n=0}^{\infty} \frac{\det( Df^{-1} \vert E_{f^{-n}(x)}^{cu})}{\det( Df^{-1} \vert E_{f^{-n}(y)}^{cu})}.
\end{equation}
\end{remark}

The limit (\ref{density}) is bounded above and away from zero by constants that depend only on $\ell$ in a neighborhood of $f$. In particular, given any $\ell\geq 1$, there are a neighborhood $\mathcal{U}$ of $f$ and $L>0$, such that for every $g \in \cU$, every $\mu \in \Gcu(g)$ and every $x \in M$, we have
\begin{equation}\label{eq:densidad}
\int \varphi \ d\mu \leq L \int_{\cQ(\ell, x)} \left( \int \varphi \ dm_D \right) \ d\hat{\mu} 
\leq  \sup_{D \in \cQ(\ell, x)} L |\mu| \int \varphi \ dm_D.
\end{equation}

Reciprocally, if we consider any $x\in M$, $\ell>0$, such that the lamination bundle $Q=Q(\ell,x)$ has  $\mu(Q)>0$, assuming that $\hat{\mu}$-almost every $D\in\cQ$, $\mu_D\prec m_D$,  we conclude that we can disintegrate $\mu_Q$ along the Pesin unstable manifolds  (which are precisely the discs belong to $\cQ$) and then $\mu$ has conditional measures absolutely continuous with respect to the Pesin unstable manifolds. But again, from Lemma~\ref{large hyperbolic set of iterate} we conclude that the union of lamination bundles $Q=Q(\ell,x)$ as above has measure bigger that $1-\epsilon$, where  $\epsilon>0$ is arbitrary. So $\mu$ must be a Gibbs $cu$-state.

The discussion above implies in particular that if $\mu$ is not a  a Gibbs $cu$-state for  $f$, then we can find a subset $X\subseteq M$ of positive $\mu$-measure (which is singular with respect to the Lebesgue measure along the center unstable manifolds) such that for every $x\in M$ and every lamination bundle $Q=Q(\ell, x)$, with associated partition $\cQ=\cQ(\ell, x)$,  we have we have $m_D(X) = 0$ for $\hat{\mu}$-almost every $D \in \cQ$.

\section{Proof of Theorem~\ref{cu is closed}} \label{proof of prop}

\subsection{A characterization of Gibbs $cu$-states}
As we outlined in the introduction, proving statistical stability of mostly expanding diffeomorphisms involves proving that a limit of Gibbs $cu$-states is a Gibbs $cu$-state. %Although this may be intuitively clear to many readers in light of Lemma~\ref{large pesin blocks}, 
A direct proof using lamination bundles (or foliated boxes) would be clumsy and difficult to make rigorous. 
Our first goal is to answer the follwing question: Which among the Gibbs $u$-states with positive Lyapunov exponents along the central direction are Gibbs $cu$-states?  We intend to answer that question  by introducing a "disintegration-free" characterization and so give a cleaner proof of Theorem~\ref{cu is closed}.

%Let us fix the notation along this section.   Let $f:M\to M$ be a $C^r$-partially hyperbolic diffeomorphism, $r>1$, and let $\mu$ be a Gibbs $u$-state such that $\mu$-almost every point $x\in M$, $\lambda^c(f,x)>0$. Let  $\hat{\lambda}^c(f,\mu):=\int\lambda^c(f,x)d\mu>0$, and fix any $0<\sigma<1$  such that $\log \sigma^{-1}<\hat{\lambda}^c(f,\mu)$.   
%For every $\ell\geq 0$, also we fix $r=r(\ell)>0$, $r/2>\delta=\delta(\ell)>0$ given by Theorem~\ref{unstable manifolds}.   This implies that there exist a center unstable lamination
%$$\cW_\ell(f) = \{W_{r}^{cu}(x): x \in \Lambda_\ell(f,\sigma)\}$$
%which allow us to build laminations bundles on every $x\in M$. For an $x\in M$ fixed, we write $Q=Q(\ell,x)$ the lamination bundle defined by \eqref{lamination bundle} with associated partition $\cQ=\cQ(\ell,x)$ defined by  \eqref{partition}. For  our fixed Gibbs $u$-state $\mu$, we denote by $\mu _Q$ the restriction of $\mu$ to $\Q$,  $\{\mu_D: D \in \cQ\}$  denotes the family of conditional probability measures  of $\mu_Q$ with respect to $Q$ and  $\hat{\mu}$ denotes the factor measure defined on $\cQ$. Recall that $\mu_Q$,  $\hat{\mu}$,  and $\{\mu_D: D \in \cQ\}$ satisfies the relation \eqref{disintegration}.

%For $\epsilon>0$, let $\ell_0=\ell_0(\epsilon)>0$ given by Remark~\ref{rmk:singlemu}. In particular, 
%$\mu(\Lambda_\ell (f,\sigma)) > 1-\epsilon$ for every  $\ell \geq \ell_0$, where $\Lambda_\ell (f,\sigma)$ is defined in \eqref{eq:deflambdan}-\eqref{eq:deflambda}. 

\begin{Theorem} \label{characterization}
Let $f:M \to M$ be a $C^r$-partially hyperbolic diffeomorphism, $r>1$, and  let $\mu$ be a Gibbs $u$-state such that $\mu$ has positive central Lyapunov exponents.  Then the following are equivalent:
\begin{enumerate}
\item $\mu$ is a Gibbs $cu$-state.
\item Given any $\epsilon>0$ and  sufficiently large $\ell$,  there exists $K>0$ such that  
\begin{equation} \label{characterizationinequality}
\int \varphi \: d\mu < \epsilon + K \cdot \sup_{W\in\cW_\ell(f)}\int \varphi \: d m_W
\end{equation}
for every continuous function $\varphi:M \to [0,1]$. Moreover, if $f$ is mostly expanding then there is a  $C^r$ neighborhood $\cU$ of $f$, such that the constant $K\geq 0$ can be chosen independent of every $g\in \cU$ and every $\mu\in\Gcu(g)$.
\end{enumerate}
\end{Theorem}

Another characterization was provided in \cite[Theorem A]{V2009} where the author  proved that an ergodic Gibbs $u$-state $\mu$  with positive central Lyapunov exponents is a Gibbs $cu$-state if and only if there is a local center-unstable manifold contained (Lebesgue-mod 0) in the basin of $\mu.$ Unfortunately, we are not able to apply that characterization here, as we have no means of finding such  center-unstable manifold contained in the basin of $\mu$.

Our characterization is rather subtle and logically intricate, but useful.  It can be stated in symbolic form like this: $\mu \in \Gcu(f) \Longleftrightarrow$
\begin{gather*}
\forall  \epsilon>0,\   \exists \ell_0>0, \  \forall \ell \geq \ell_0, \ \exists K>0 ,\  \forall \varphi \in C^0(M, [0,1]),  \ \exists W \in \cW_\ell(f): \\
\int \varphi \ d\mu < \epsilon + K \int \varphi \ dm_W.
\end{gather*}

It is an expression of quantifier rank equal to six and must be dealt with very carefully. We believe that it reflects an inherent intricacy of the notion of Gibbs $cu$-states (or SRB measures more generally) which is not always apreciated. It also explains why a carefully written proof of  convergence of Gibbs $cu$-states is harder than one may think.

\begin{proof}[Proof that (i) implies (ii) in Theorem~\ref{characterization}]

Suppose that $\mu \in \Gcu(f)$ and fix some $\epsilon>0$ and $\ell \geq \ell_0 $. % Choose $\ell_0\geq 1$ as in Remark~\ref{rmk:singlemu} and let $\ell \geq \ell_0 $. Choose $\delta>0$ small enough so that, given any $x \in M$,  $Q(\ell,  \delta; x)$ is a lamination bundle with associated partition $\cQ(\ell, \delta; x)$. 
Choose $x_1, \ldots, x_k$, $k=k(\ell)\geq 1$, such that  $\{B_\delta(x_i): i =1, \ldots, k \}$ is a cover of $M$ . Let $\cQ_i = \cQ(\ell; x_i)$ and $Q_i = Q(\ell; x_i)$ for $1 \leq i \leq k$.

For every $i \in \{1, \ldots, k \}$, we denote $\mu_i$  the restriction of  $\mu $ to $Q_i$. Denotes by $\{\mu^i_D: D \in \cQ_i\}$ the family of conditional measures of $\mu_i$ with respect to the partition $\cQ_i$ and by $\hat{\mu}_i$ the factor measures defined on $\cQ_i$. %Recall that we may decompose  $\mu_i$ with respect to the partition $\cQ_i$, i.e. to find a measurable family of probability measures $\{\mu_D: D \in \cQ_i\}$ and a factor measure $\hat{\mu}_i$ on $\cQ_i$ with $|\hat{\mu}_i| = \mu(Q_i) $ satisfying \eqref{disintegration}.

Let $L=L(\ell)>0$ be as in the Remark~\ref{rmk:densidad} such that for every $i \in \{1, \ldots, k \}$ and $\hat{\mu}_{i}$-almost every $D \in \cQ_i$, the density of $\mu^i_D$ with respect to $m_D$ is bounded above by $L$. Then
\begin{equation*}
\int_D \varphi \ d\mu^i_D \leq L \int_D \varphi \ dm_D
\end{equation*}
for $\hat{\mu}_{i}$-almost every $D \in \cQ_i$.  Let $S=S(\ell)$ be an upper bound for $\{ |\vol_D|^{-1}: D \in \cQ_i, 1\leq i \leq k \}$ and let $T=T(\ell)$ be an upper bounded for $\{ |\vol_W|: W \in \cW_\ell(f) \}$. 

On the one hand we have
\begin{eqnarray}
\int_{Q_i} \varphi \ d\mu & = &\int_{\cQ_i} \left( \int_D \varphi \ d\mu^i_D \right) \ d\hat{\mu}_i (D) \nonumber \\
& \leq& \hat{\mu_i}(\cQ_i)  L \sup_{D \in \cQ_i} \int_D\varphi \ dm_D \nonumber \\
& = &\mu(Q_i) L \sup_{D \in \cQ_i} \int_D \varphi \ d\frac{\vol_D}{|\vol_D|} \nonumber  \\
& \leq& L S \sup_{D \in \cQ_i} \int \varphi \ d\vol_D. \label{eq:desigualdad45}
\end{eqnarray}

On the other hand, since for every $D\in\cQ_i$ there is a $W\in\cW_\ell(f)$ such that $D\subseteq W$, and the rank of $\varphi$ is the unit interval, then 

\begin{eqnarray}
\sup_{D \in \cQ_i} \int_{Q_i} \varphi \ d\vol_D& \leq &\sup_{W \in \cW_\ell(f)}\int_{W} \varphi d\vol_W \nonumber \\
& \leq &\sup_{W \in \cW_\ell(f)}\int_{W} \varphi |\vol W|\frac{\ d\vol_W}{|\vol W|}\nonumber \\
& =& \sup_{W \in \cW_\ell(f)}\int_{W} \varphi |\vol W| \ dm_W \nonumber\\
& \leq& T  \sup_{W \in \cW_\ell(f)} \int \varphi \ dm_W. \label{eq:desigualdad46}
\end{eqnarray}

Notice that $\Lambda_\ell(f) \subset Q_1 \cup \ldots \cup Q_k$ and that $\mu(\Lambda_\ell(f))>1-\epsilon$. So from \eqref{eq:desigualdad45} and \eqref{eq:desigualdad46} we have
we obtain,
\begin{eqnarray*}
\int \varphi \ d\mu & <& \epsilon + \sum_{i=1}^k \int_{Q_i} \varphi \ d\mu \\
&\leq & \epsilon+ k L S   \sup_{D \in \cQ_1\cup \ldots \cup \cQ_k} \int \varphi \ d\vol_D \\
& \leq &\epsilon + k L S T  \sup_{W \in \cW_\ell(f)} \int \varphi \ dm_W
\end{eqnarray*}

The proof follows by taking $K=k L ST$.

For the second part of the statement, we note that $K=K(\ell)>0$ depends on  the cardinality $k=k(\ell)\geq 1$ of the finite covering by balls of radius $\delta(\ell,f)>0$,  and the bounds $S=S(\ell)>0$ and $T=T(\ell)$  and the bound $L=L(\ell)>0$ provides in Remark~\ref{rmk:densidad}. All this constants can be chosen uniformly in a neighborhood $\cU$ of $f$, independent of the measure $\mu\in\Gcu(g)$, $g\in\cU$,  according to Lemma~\ref{large pesin blocks} and Theorem~\ref{unstable manifolds}.
\end{proof}

The converse statement in Theorem~\ref{characterization} is harder to prove. We need an auxiliary result. Recall that if $Q$ is a lamination bundle with associated partition $\cQ$, such that $\mu(Q)>0$, we  define the Borel measure $m_Q$ on $Q$ as 
$$\int\varphi\, dm_Q=\int_\cQ\int_D\varphi\,d m_D d\hat{\mu}(D),$$
for every $\varphi\in C^0(M,\mathbb{R})$.

\begin{lemma} \label{separating function}
Let  $Q=Q(\ell, x)$ be a lamination bundle with associated partition $\cQ = \cQ(\ell, x)$,  $\mu$ be a Borel measure such that $\mu(Q)>0$ and  $\phi:M \to [0,1]$  a continuous function. Then, given any $b>0$ such that $b^2 > \int \phi \ d m_Q$, there exists a continuous function $\xi:M \to [0,1]$ with $1-\xi$ supported in $B_{2\delta}(x)$ such that
\begin{itemize}
\item[{i.}] $\mu(\{x: \xi(x)<1 \}) < b$; and \\
\item[{ii.}] $\displaystyle \int_D  \phi\cdot \xi \ dm_D < b$, for every $D \in \cQ$.
\end{itemize}
\end{lemma}

\begin{proof} 

Let $\mathcal{K}$ be the set of those $D \in \cQ$ for which $\int_D \phi \ dm_D \geq b$  and let $K= \bigcup_{D \in \mathcal{K}} D$. 
By Chebyshev's inequality we have
\begin{equation*}
\mu(K) = \hat{\mu}(\mathcal{K}) \leq \frac{1}{{b}} \int_{\mathcal{K}} \left( \int_D \phi \ dm_D \right) \ d\hat{\mu}(D)
\leq  \frac{1}{{b}}\int \phi \ dm_Q < {b}.
\end{equation*}

Fix $0< \tau < 2 \delta$ large enough so that $m_D(D \setminus \overline{B_\tau(x)}) < b$ for every $D \in \cQ$ and let $K_\tau = K \cap \overline{B_\tau (x)}$.  Thus in particular
\[m_D(D \cap K_\tau^c) < b \]
for every $D \in \cK$.

It follows from (ii) in Theorem~\ref{unstable manifolds} that the map $x \mapsto \int \phi \ dm_{D(x)}$ is continuous on $Q$, where $D=D(x)$ is the element of $\cQ$ that contains $x$. Thus $K_{\tau}$ is closed. Note that $\mu(K_\tau) \leq \mu(K) < b$.
Let $U$ be an open neighborhood of $K_{\tau}$ with $\overline{U} \subset B_{2\delta}(x)$  such that $\mu(U) < b$. Such $U$ can always be found due to the regularity of Borel measures.
Let $\xi: M \to [0,1]$ be a continuous function such that $\xi=0$ on $K_\tau$ and $\xi = 1$ on the compliment of $U$. Then 
$$\mu(\{x: \xi(x)<1 \})\leq \mu(U) <b.$$

We claim that $\int \phi\cdot\xi\ dm_D < b$ for every $D \in \cQ$. Indeed, if $D \notin \mathcal{K}$, then 
$$\int \phi\cdot\xi\ dm_D \leq \int \phi \ dm_D<{b}$$ 
since $0\leq \xi \leq 1$ and by the  definition of $\mathcal{K}$. 
On the other hand, if  $D \in \mathcal{K}$, then  
\begin{align*}
\int \phi\cdot\xi\ dm_D & \leq \int \xi \ dm_D \\ & = \int_{D\cap K_\tau} \xi \ dm_D + \int_{D\cap K_\tau^c} \xi \ dm_D \\ & \leq 0 + m_D(D \cap K_\tau^c) < b,
\end{align*}
since $0 \leq \phi \leq 1$ and $\xi$ vanishes at $K_\tau$.
\end{proof}

\begin{proof}[Proof that (ii) implies (i) in Theorem~\ref{characterization}]
Suppose that $\mu$ is not a Gibbs $cu$-state. 
Then there is some measurable set $X \subset M$ with $\mu(X)>0$ for which the following happens: For every $x\in X$ and every  lamination bundle $Q=Q(\ell, \sigma, x)$, 
with associated partition $\cQ=\cQ(\ell,\sigma,x)$,  we have we have $m_D(X) = 0$ for $\hat{\mu}$-almost every $D \in \cQ$. %, where $\mu _Q$ denotes  the restriction of $\mu$ to $\Q$,  $\{\mu_D: D \in \cQ\}$  denotes the family of conditional probability measures  of $\mu_Q$ with respect to $Q$ and  $\hat{\mu}$ denotes the factor measure defined on $\cQ$. Recall that $\mu_Q$,  $\hat{\mu}$,  and $\{\mu_D: D \in \cQ\}$ satisfies the relation \eqref{disintegration}.

Let $\epsilon = \mu(X)/2>0$. We must prove that, given any integer $\ell_0\geq 1$, there exists $\ell \geq \ell_0$ with the property that for every $K>0$ it is possible to find a continuous function $\varphi:M \to [0,1]$ such that
\begin{equation} \label{must prove}
\int \varphi \ d\mu \geq \epsilon + K \int \varphi \ dm_W.
\end{equation}
for every $W \in \cW_\ell(f)$.
To this end, fix $\ell_0>0$ arbitrarily. Thereafter choose $\ell \geq \ell_0$ large enough so that $\mu(\Lambda_\ell(f, \sigma))>1-\epsilon$. 
Choose $x\in M$ such that $\mu(B_{\delta} (x)\cap X))>0$ and  fix  $Q'= Q(\ell,  \sigma^{1/2} ; x)$ and $\cQ' = \cQ (\ell,  \sigma^{1/2} ; x)$. Moreover,  let $\hat{\mu}_{\cQ'}$ be the factor measure of $\mu_{Q'}$ with respect to the partition $\cQ '$ and denote by $m_{Q'}$ the measure $\int_{\cQ'} m_D \ d\hat{\mu}_{\cQ'}(D)$. 

Fix $K>0$ arbitrarily and choose $0<\alpha< \epsilon/3$ such that

$$0<\alpha<\left(\frac{\mu(X)}{2(K+2)}\right)^2.$$

 Note that $\mu_{Q'} \vert (X\cap B_{\delta} (x))$ and $m_{Q'}$ are mutually singular measures. Therefore we can find a continuous function $\phi:M \to [0,1]$ such that  $\supp\phi\subseteq B_{\delta}(x)$,  satisfying
\begin{equation}\label{eq:disin1}
m_{Q'}(X)=0\leq\int\phi\ d m_{Q'}<\alpha<\mu(X)-\alpha<\int\phi\ d\mu.
\end{equation}
We can apply Lemma~\ref{separating function} to  $\phi$ taking $b=\frac{\mu(X)}{2(K+2)}>0$, and we find a continuous function $\xi:M \to [0,1]$ with $1-\xi$ supported in $B_{2\delta}(x)$ satisfying the following: From Lemma~\ref{separating function} item i, and since $0\leq\phi\leq 1$ and $0\leq \xi\leq 1$ we have, 
$$ \int\phi(1-\xi)\ d\mu \leq \int (1-\xi)\ d\mu \leq \mu(\{x\::\: 1-\xi(x)>0\})<b.$$

Combining this last relation with \eqref{eq:disin1} we obtain  
$$
\mu(X)-\alpha < \int\phi \,  d\mu\leq\int \phi\, \xi\ d\mu+\int\phi (1-\xi)\ d\mu< \int \phi\, \xi\ d\mu+b
$$
and since $\alpha<b$,then
\begin{equation}\label{eq:upperbound}
\left(\frac{K+1}{K+2}\right)\mu(X)=\mu(X)-2b< \mu(X)-\alpha-b < \int\phi\, \xi \  d\mu.
\end{equation}
On the other hand, Lemma~\ref{separating function}, item ii., implies immediately that  for every $D\in \cQ'$,
\begin{equation}\label{eq:lowerbound}
\int\phi\,\xi\ dm_D< \frac{\mu(X)}{2(K+2)}
\end{equation}
Now, calling $\varphi= \phi\,\xi$ and combining \eqref{eq:upperbound} and \eqref{eq:lowerbound} we have
\begin{equation}\label{eq:finalbound}
K\int\varphi\ dm_D+\epsilon< K\left(\frac{\mu(X)}{2(K+2)}\right)+\frac{\mu(X)}2=\left(\frac{K+1}{K+2}\right)\mu(X)< \int\varphi\  d\mu.
\end{equation}

Let $W \in \cW_\ell(f)$. If $W \cap B_\delta(x) = \emptyset$, then the right hand side of (\ref{must prove}) vanishes, so there is nothing to prove. Now suppose that  $W \cap B_\delta(x) \neq \emptyset$. From Theorem \ref{unstable manifolds} item (iii), we know that $W \subset \Lambda_{\ell}(f, \sigma^{1/2})$. Therefore, there exists $D \subset \cQ'$ such that $W \cap B_{2\delta}(x) \subset D$. In this case  (\ref{must prove}) follows from \eqref{eq:finalbound}. 

\end{proof}

\subsection{Concluding the proof of Theorem~\ref{cu is closed}}

Everything done so far in section \ref{proof of prop} (briefly speaking, our Pliss-like Lemma and our characterization of Gibbs $cu$-states) have been for the purpose of proving Theorem~\ref{cu is closed}. The proof is based on the observation that the quantities $\ell$ and $K$ in Theorem~\ref{characterization} are uniform in a neighborhood of $f$.

Let $f$ be a $C^r$ mostly expanding diffeomorphism, $r>1$. Consider a sequence $(f_n, \mu_n) \in \Gcu(\mathcal{U}_{\mathcal{ME}})$ such that $f_n$ converges in $C^r$ to some mostly expanding diffeomorphism $f$ and $\mu_n$ converges weakly* to some measure $\mu$. To prove Theorem~\ref{cu is closed} we must establish that $\mu$ is a Gibbs $cu$-state. We do that by showing that the inequality (\ref{characterizationinequality}) passes to the limit. The following straightforward lemma is useful.

\begin{lemma}\label{closeness of pesin blocks}
Let $f:M \to M$ be a $C^r$, $r>1$,  mostly expanding diffeomorphism, $\ell \in \mathbb{N}$, and suppose that $f_n$ is a sequence of $C^r$ mostly expanding diffeomorphisms converging to $f$. Then any accumulation point of $\Lambda_\ell(f_n)$ belongs to $\Lambda_\ell(f)$. That is, 
\begin{equation*}
\bigcap_{n} \overline{\bigcup_{k\geq n} \Lambda_\ell(f_n)} \subset \Lambda_\ell(f).
\end{equation*}
\end{lemma}

\begin{proof} Let us fix $\ell\geq 1$ and $N\geq 1$. Then, just for  continuity of $(x,f) \to \|Df^{-\ell}\vert E^{cu}(y)\|$, for $\eta>0$ there exist a neighborhood $\mathcal{U}_N$ of $f$, and $\epsilon_N>0$  such that if $g\in \mathcal{U}_N$ and  $x,y\in M$ satisfies  $\di(x,y)<\epsilon_N$, then
\begin{equation}\label{eq:cota4.3.1}
\prod_{j=0}^{N-1} \|Df^{-\ell}\vert E_{f^{-\ell j}(x)}^{cu}\| \prod_{j=0}^{N-1} \|Dg^{-\ell}\vert E_{g^{-\ell j}(y)}^{cu}\|^{-1} < 1+\eta.
\end{equation}

Since $f_n\to f$, we can assume that there exist $n\geq 1$ such that, for every $k\geq n$, $f_k\in\mathcal{U}_N$. If we assume  that $x\in \overline{\bigcup_{k\geq n} \Lambda_\ell(f_n)}$, for every $n\geq 1$, then for  $\epsilon_N>0$ there exist $k\geq n$ and $y\in \Lambda_\ell(f_k) $ such that $\di(x,y)<\epsilon_N$. Since $y\in \Lambda_\ell(f_k) $, then 
\begin{equation}\label{eq:cota4.3.2}
\prod_{j=0}^{N-1} \|Df_k^{-\ell}\vert E_{f_k^{-\ell j}(y)}^{cu}\| \leq \sigma^{\ell N}.
\end{equation}
On the other hand, since  $\di(x,y)<\epsilon_N$, it follows from \eqref{eq:cota4.3.1} and \eqref{eq:cota4.3.2} that
\begin{equation}\label{eq:cota4.3.3}
\prod_{j=0}^{N-1} \|Df^{-\ell}\vert E_{f^{-\ell j}(x)}^{cu}\| \leq (1+\eta)\sigma^{\ell N}
\end{equation}

Since $\eta>0$  and $N\geq 1$ are arbitrary, then $x\in \Lambda_\ell(f)$ as we claim.

\end{proof}

\begin{proof}[Ending the proof of Theorem~\ref{cu is closed}]

Fix $\epsilon >0$ arbitrarily and let $\ell_0\geq 1$ be as in Lemma~\ref{large pesin blocks}. Fix $\ell \geq \ell_0 $ arbitrarily. Since $f_n$ converges to $f$, we have $\mu_n(\Lambda_\ell(f_n))> 1-\epsilon$ for sufficiently large $n$, say $n \geq n_0$. Also, we fix $r=r(\ell)>0$ and  $\delta(\ell)>0$ as in  Theorem~\ref{unstable manifolds}. 

%We can find $k=k(\ell)\geq 1$ and $x_1,\dots x_k\in M$, such that $\{B(x_i,\delta):i=1,\dots k\}$ is a finite covering of $M$.  Denote by $Q_i=Q(\ell, x_i)$, and $\cQ=(\ell, x_i)$ their respective lamination bundle. Note that $\Lambda_\ell(f)\subseteq\cup_{i=1}^k Q_i$.
%For every $i \in \{1, \ldots, k \}$, we denote $\mu_i$  the restriction of  $\mu $ to $Q_i$. Denotes by $\{\mu_D: D \in \cQ_i\}$ the family of conditional measures of $\mu_i$ with respect to the partition $\cQ_i$ and by $\hat{\mu}_i$ the factor measures defined on $\cQ_i$. 

According to Theorem~\ref{characterization}, it suffices to prove that there exists $K>0$ such that, given any continuous function $\varphi : M \to [0,1]$, we have
\begin{equation*}
\int \varphi \ d\mu_n < \epsilon + K \sup_{W \in \cW_\ell(f)} \int \varphi \ d m_W.
\end{equation*}
For every $x \in M$, consider the sets
\begin{equation*}
\cQ_n(\ell, x) = \{W_{r}^{cu}(y)\cap B_{2 \delta} (x): y \in \Lambda_\ell(f_n) \cap \overline{B_\delta (x)} \}
\end{equation*}
and 
\begin{equation*}
Q_n(\ell,  x) = \bigcup_{D \in \cQ_n(\ell,x)} D.
\end{equation*}
Choose points $x_1, \ldots, x_k$, $k=k(\ell)>0$, so that $M \subset B_\delta(x_1)\cup \ldots \cup B_\delta(x_k)$. Write $Q^i_n = Q^i_n(\ell, x_i)$ and $\cQ_n^i = \cQ^i_n(\ell, x_i)$ for $1 \leq i \leq k$ and  we denote $\mu_n^i$  the restriction of  $\mu_n $ to $Q_n^i$. Denote by $\{\mu^i_{n,D}: D \in \cQ_n^i\}$ the family of conditional measures of $\mu_n^i$ with respect to the partition $\cQ_n^i$ and by $\hat{\mu}_n^i$ the factor measures defined on $\cQ_n^i$. 

%\begin{equation*}
%\mu_n \vert Q_i^n = \int_{\cQ_i^n} \left( \int \varphi \ d\mu_D \right) \ d\hat{\mu}_i^n(D)
%\end{equation*}
%be the Rokhlin disintegration of the restriction of $\mu_n$. 

Let $L>1$ be such that, for every $n\geq n_0$, every $i \in \{1, \ldots, k \}$ and $\hat{\mu}_n^i$-almost every $D \in \cQ_n^i$, the density of $\mu^i_{n,D}$ with respect to $m_D$ is bounded above by $L$. Let $K=k L$. Now consider any continuous function $\varphi :M \to [0,1]$. For every $n>n_0$ and every $i \in \{1, \ldots, k\}$ we have
\begin{equation*}
\int_{Q_n^i} \varphi \ d\mu_n = \int_{\cQ_n^i} \left( \int \varphi \ d\mu^i_{n,D} \right) d\hat{\mu}_n^i(D) \leq L \mu(Q_n^i) \sup_{D \in Q^i_n} \int \varphi \ dm_D 
\end{equation*}
 Notice that $\mu_n(Q_n^i)\geq\mu_n(\Lambda_\ell(f_n) \cap B_{ \delta} (x_i))$ for every $1\leq i \leq k$ and every $n \geq n_0$. Since $M = B_{ \delta}(x_1) \cup \ldots \cup B_{ \delta}(x_k)$, this gives us the estimate $\mu_n(Q_n^1 \cup \ldots \cup Q_n^k)>1-\epsilon$. It follows that 
\begin{equation*}
\int \varphi \ d\mu_n < \epsilon + \sum_{i=1}^k L \mu(Q_n^i) \sup_{D \in \cQ_n^i} \int_D \varphi \ dm_D
\leq L k \max_{1\leq i \leq k} \sup_{D \in \cQ_n^i} \int \varphi \ dm_D.
\end{equation*}
 For every $n>n_0$, choose a disk $D_n \in \cQ_n^1 \cup \ldots \cup \cQ_n^k$ such that 
\begin{equation*}
\int \varphi \ dm_{D_n} = \max_{1\leq i \leq k} \sup_{D \in \cQ_n^i} \int \varphi \ dm_D .
\end{equation*}
Each disk $D_n$ is contained in some $W_n \in \cW_\ell(f_n)$. By Lemma~\ref{closeness of pesin blocks}, $W_n$ accumulates on some disk $W \in \cW_\ell(f )$. Therefore
\begin{equation*}
\int \varphi \ d\mu \leq \epsilon + k L \int \varphi \ dm_W. 
\end{equation*}
In particular,
\begin{equation*}
\int \varphi \ d\mu < \epsilon + K \cdot \sup_{W \in \cW_{\ell}(f)} \int \varphi \ dm_W,
\end{equation*}
where $K = kL+1$ is independent of $\varphi$.
\end{proof}

\section{Proof of main theorems}\label{proof several pm}

We are now in a position to prove Theorems~\ref{one pm},~\ref{several pm}, and \ref{weak stability}. 

\subsection{Proof of Theorem~\ref{weak stability}}
 Let $f$ be mostly expanding. By Theorem~\ref{MACA1teoC} there are finitely many physical measures $\nu_1, \ldots \nu_k$. By the definition of weakly statistically stability, it suffices to prove that  given any  sequence of diffeomorphisms $f_n$ converging to $f$ in the $C^r$ topology, $r>1$, and any sequence of physical measures $\mu_n$ of $f_n$, if  $\mu$ is an accumulation point of $\mu_n$, then $\mu$ is a convex combination of $\nu_1, \ldots \nu_k$. 
  
  Let $f_n$ be a sequence of diffeomorphisms converging to $f$ in the $C^r$ topology for some $r>1$. Theorem~\ref{MACA1teoC} again, implies that mostly expanding is an open condition. Then,  we may  suppose that each $f_n$ is mostly expanding. Let $\mu_n$ be a sequence of physical measures for $f_n$ respectively. Proposition~\ref{pm vs cu} implies that each $\mu_n$ is a Gibbs $cu$-state. Upon possibly taking a subsequence, we may suppose that the sequence $\mu_n$ converges to some measure $\mu$. Theorem~\ref{cu is closed} tells us that $\mu$ is a Gibbs $cu$-state. We know that (Proposition~\ref{prop:cuM2}) Gibbs $cu$-states are convex combinations of ergodic Gibbs $cu$-states and, again by Proposition~\ref{pm vs cu}, that ergodic Gibbs $cu$-states are physical measures. It follows that $\mu$ is a convex combination of physical measures of $f$. This completes the proof of Theorem~\ref{weak stability}.

\subsection{Proof of Theorem~\ref{several pm}}
The proof of  upper semi-continuity of the number of physical measures is by contradiction. Suppose, therefore, that upper semi-continuity on the number of physical measures does not hold. That means that there exists some mostly expanding diffeomorphism $f$, and a sequence $f_n \rightarrow f$ of mostly expanding diffeomorphisms converging to $f$ in the $C^r$ topology, all of which have a number of physical measures larger than that of $f$. In other words, if the physical measures of $f$ are $\mu^j, \ j \in J$ for some finite set $J$, there are measures $\nu_n^i, \ i\in I$ for some finite set $I$ with $|I| = |J|+1$ such that 
\begin{enumerate}
\item $\nu_n^i$ is a physical measure for $f_n$ for every $n$,
\item $\nu_n^i \neq \nu_n^{i'}$ for every $n$ and every $i,i' \in I$, with $i \neq i'$.
\end{enumerate}
Upon taking  an appropriate subsequence, we may also assume from Theorem~\ref{weak stability} that
\begin{enumerate}[resume]
\item for each $i \in I$, there exist non-negative numbers $\alpha_{i, j}, \ j \in J$ with $\sum_{j \in J} \alpha_{i,j} = 1$, such that 
$\nu_n^i \to \sum_{j \in J} \alpha_{i, j} \mu^j$. 
\end{enumerate}  

To get a contradiction, we shall prove that each column in the matrix $(\alpha_{i,j})$ can have at most one positive element. Since the number of rows is larger than the number of columns, this implies that $(\alpha_{i,j})$ must have a row of zeros, contradicting $\sum_{j \in J} \alpha_{i,j} = 1$.

To see why each column of $\alpha_{i,j}$ can have at most one positive element, let $P = \{(i,j) \in I \times J : \alpha_{i,j}>0 \}$ and $\alpha = \min \{ \alpha_{i,j}: (i,j) \in P \}$. Recall that each point in the Pesin block $\Lambda_\ell(f_n)$ has an unstable manifold of a fixed size $r=r(\ell)>0$. Moreover, from Lemma~\ref{large hyperbolic set of iterate} we know that it is possible to choose $\ell$ such that $\nu_n^i(\Lambda_\ell(f_n))>1-\alpha/2$ for every large $n$.  

 The angle between $E^s$ and $E^{cu}$ is bounded away from zero in a robust manner. Hence there is some $\rho>0$ such that for every large $n$, and any point $x\in \Lambda_\ell(f_n)$, the set $\Gamma(f_n, \ell, x) = \bigcup_{y \in W_{r}^{cu}(x)} W^s(f_n, y)$ contains the ball $B_\rho(x)$.
 
 We cover the supports of $\mu^j$ by balls $B_{\rho/2}(x_{j,k})$, $k \in K$, where $K$ is some finite set. For sufficiently large $n$ we have
 \begin{enumerate}
 \item[{(iv)}] $\nu_n^i(B_{\rho/2} (x_{j, k})) > 0 $ for every $(i,j, k) \in P \times K$, and
 \item[{(v)}] $\nu_n^i (\Lambda_\ell(f_n)  \cap B_{\rho/2} (x_{j,k}) ) > 0$ for every $(i,j) \in P$ and some $k \in K$.
 \end{enumerate}
 
 Thus given any $i\in I$ choose $j \in J$ such that $\alpha_{i,j} > 0$ and $k \in K$ such that (v) holds. Since  $\nu_n^i$ is an ergodic Gibbs $cu$-state, there is some $x \in B_{\rho/2}(x_{j, k})$ such that $B(\nu_n^i)$ has full leaf volume in $W_{r}^{cu}(x)$. Therefore, by absolute continuity of the stable foliation, $B(\nu_n^i)$ has full volume in $\Gamma(f_n,\ell, x)$. By our choice of $\rho$, we have $\Gamma(f_n,\ell, x) \supset B_{\rho}(x) \supset B_{\rho/2}(x_{j, k})$ so, in particular, $B(\nu_n^i)$ has full volume in $B_{\rho/2}(x_{i,k})$.

Now take any $i' \in I$ different from $i$. We claim that $(i',j) \notin P$. Indeed, if it were not so, then by (iv) we would have $\nu_n^{i'}(B_{\rho/2}(x_{j,k}))>0$ for sufficiently large $n$. Therefore, there would be some $\ell'$ and some $x' \in B_{\rho/2}(x_{j,k})$ such that $B(\nu_n^{i'})$ has full leaf volume in $W_{r(\ell')}^{cu}$. Again, by absolute continuity of the stable foliation, that would imply that $B(\nu_n^{i'})$ has positive volume in $B_{\rho/2}(x_{j,k})$. But that is absurd, since $B(\nu_n^i)$ has full volume in $B_{\rho/2}(x_{j,k})$ and basins of distinct physical measures are disjoint. Thus we have proved that each column in the matrix $(\alpha_{i,j})$ has at most one non-zero entry  and the proof of upper semi-continuity of the number of physical measures is complete. 
 
It remains to prove statistical stability in its most general setting. To this end, we will use the  claimed proved above: there is a neighborhood of $f$ where the number of physical measures is constant. Then we can suppose that $f_n$ is a sequence of mostly expanding diffeomorphisms converging to a mostly expanding diffeomorphism $f$ and that each $f_n$ and $f$ all have the same number of physical measures. 
 We use the notation above, so that the physical measures of $f$ are $\mu^j$, $j \in J$ and those of $f_n$ are $\nu_n^i$, $i \in I$. The difference now is that $|I| = |J|$. By taking subsequences we may assume that $\nu_n^i \rightarrow \sum_{j \in J} \alpha_{i,j} \mu^j$ for some non-negative numbers $\alpha_{i,j}$ with $\sum_{j \in J} \alpha_{i,j} = 1$. It was proved above that in this case, each column of the $|I| \times |J|$ matrix $A = (\alpha_{i,j})_{(i,j) \in I\times J}$ has at most one positive element. Now, $A$ is a square matrix and the sum of the entries in each row is $1$. In particular each row has at least one positive entry. Therefore $A$ must be a permutation matrix, i.e. one for which each column and each row has exactly one entry equal to $1$ and all other entries are zero. Define the map $\tau: I \to J$ so that $\tau(i)$ is the unique element of $A$ such that $\alpha_{i,j} = 1$. Then $\nu_i^n$ converges to $\mu_{\tau(i)}$ for every $i \in I$. That completes the proof of statistical stability.

\subsection{Proof of Theorem~\ref{one pm}}

Let $f:M \to M$ be as in Theorem~\ref{one pm}. Then Theorem~\ref{MACA1teoC} says that $f$ has a finite number of physical measures, whose basin of attraction cover Lebesgue almost every point in $M$. Now, since $f$ is mostly expanding, the basin of each physical measure is open, up to a zero Lebesgue measure set (see e.g. \cite[Lemma 4.5]{AV2018}). Thus it follows from the assumption of transitivity that $f$ has exactly one physical measure. Now, according to Theorem~\ref{several pm}, the number of physical measures varies upper semi-continuously on $f$. Hence there is a $C^r$ neighborhood $\cU$ of $f$ such that every $g \in \cU$ has a unique physical measure $\mu_g$ whose basin has full Lebesgue measure in $M$. Moreover, Theorem ~\ref{several pm} implies that the map $ \Diff(M) \ni g \mapsto \mu_g \in \M^1(M)$ is indeed continuous. 

\vspace{1cm}
\noindent {\bfseries Acknowledgements} We thank the reviewers for them thorough review and highly appreciate the comments and suggestions, which significantly contributed to improving the quality of the publication.

\bibliographystyle{plain}

\bibliography{MACA2}

\end{document}